\documentclass[11pt]{article}
\usepackage{amsmath,amssymb,cite,bbm,amstext,parskip,url,graphicx,algorithm,tikz,float}
\usepackage{lipsum}
\usepackage{amsfonts}
\usepackage{epstopdf}
\usepackage{algcompatible}
\usepackage{subcaption}
\usepackage{enumitem}
\usepackage{tablefootnote}
\usepackage{adjustbox}
\usepackage{graphics}
\usepackage{bbold}
\usepackage{amsthm}
\usepackage[margin=1in]{geometry}
\ifpdf
  \DeclareGraphicsExtensions{.eps,.pdf,.png,.jpg}
\else
  \DeclareGraphicsExtensions{.eps}
\fi
\usepackage{array}

\makeatletter
\graphicspath{{figure/},{figure_revision/},{figure/revision/},{./}}

\newtheorem{theorem}{Theorem}
\newtheorem{lemma}[theorem]{Lemma}
\newtheorem{proposition}[theorem]{Proposition}

\newtheorem{remark}{Remark}

\newtheorem{assumption}{Assumption}

\newcommand{\blue}{\color{black}}
\newcommand{\black}{\color{black}}

\newcommand{\At}{\calA^{(t)}}
\newcommand{\As}{\calA^{(s)}}
\newcommand{\Ts}{T^{(s)}}
\newcommand{\Tt}{T^{(t)}}
\newcommand{\calL}{\mathcal{L}}
\newcommand{\calW}{\mathcal{W}}
\newcommand{\calG}{\mathcal{G}}
\newcommand{\calV}{\mathcal{V}}
\newcommand{\calE}{\mathcal{E}}
\newcommand{\calP}{\mathcal{P}}

\newcommand{\calA}{\mathcal{A}}
\newcommand{\calO}{\mathcal{O}}
\newcommand{\xp}{x^{(p)}}
\newcommand{\yst}{y^{(s,t)}}

\newcommand{\jn}{\text{join}}
\newcommand{\cs}{\text{cross}}
\newcommand{\bbeta}{\boldsymbol{\beta}}

\newenvironment{keywords}{%
\vspace{1ex}\noindent\small
\textbf{Keywords:}\ }%
{\par\vspace{1ex}}

\newcolumntype{M}[1]{>{\centering\arraybackslash}m{#1}}
\newcolumntype{R}[1]{>{\raggedleft\arraybackslash}m{#1}}

\title{A lasso-alternative to Dijkstra's algorithm for identifying short paths in networks}

\author{%
  Anqi Dong\thanks{Division of Decision and Control Systems and Department of Mathematics, KTH Royal Institute of Technology, SE-100 44 Stockholm, Sweden (\texttt{anqid@kth.se}).}%
  \and Amirhossein Taghvaei\thanks{Department of Aeronautics and Astronautics, University of Washington, Seattle, WA 98195, USA (\texttt{amirtag@uw.edu}).}%
  \and Tryphon T. Georgiou\thanks{Department of Mechanical and Aerospace Engineering, University of California, Irvine, CA 92697, USA (\texttt{tryphon@uci.edu}).}%
}

\begin{document}

\maketitle

\begin{abstract}
We revisit the problem of finding the shortest path between two selected vertices of a graph and formulate this as an $\ell_1$-regularized regression -- Least Absolute Shrinkage and Selection Operator (lasso). We draw connections between a numerical implementation of this lasso-formulation, using the so-called LARS algorithm, and a more established algorithm known as the bi-directional Dijkstra. Appealing features of our formulation include the applicability of the Alternating Direction of Multiplier Method (ADMM) to the problem to identify short paths, and a relatively efficient update to topological changes.
\end{abstract}



\section{Introduction}
The problem of finding the shortest path between two vertices of a graph has a long history \cite{wiener1873ueber,tarry1895probleme} and a wide range of applications \cite{waxman1988routing,mortensen1995intelligent}. A classical algorithm to determine a shortest path is due to Dijkstra \cite{dijkstra1959note}. Since Dijkstra's early work, a variety of alternative methods have been developed to reduce complexity and address variants of the problem~\cite{bast2016route,ahuja1990faster,chandrasekar1994self,van1976design,fredman1987fibonacci}. A salient issue in applications involving graphs of considerable size, which motivated the present work, is that identifying a shortest path is not absolutely essential, whereas identifying a reasonably short path may suffice~\cite{waxman1988routing,potamias2009fast}.

Driven by such considerations and inspired by the effectiveness of convex optimization techniques to address large-scale problems~\cite{boyd2004convex,boyd2011distributed}, we introduce a formulation of the shortest path problem as an $\ell_1$-regularized regression, known as the {\em Least Absolute Shrinkage and Selection Operator }({\em lasso})~\cite{tibshirani1996regression}. This type of regularization/relaxation is ubiquitous in inverse problems throughout engineering, statistics, and mathematics, with a rich library of numerical implementations for high-dimensional problems -- an added incentive for the approach that we advocate.

Thus, a main contribution of this work is to formulate the shortest path problem as $\ell_1$-regularized regression, a convex optimization problem \cite{tibshirani2013lasso,tibshirani1996regression}. This formulation is, to the best of the authors' knowledge, original.

A second contribution stems from exploring at depth a popular $\ell_1$-regularized regression solver, known as {\em Least Angle Regression} ({\em LARS}), as applied to the shortest path problem. Specifically, we have shown (Theorem \ref{thm:thm1}) that the LARS implementation of our ``lasso-shortest-path'' formulation replicates a defining feature of the so-called {\em bi-directional Dijkstra algorithm}, to iteratively build two shortest-path trees, starting from the two specified vertices and until the two trees connect.
Through this connection, we present a new perspective of Dijkstra's algorithm that is completely different than the common presentation as a greedy algorithm or a dynamic programming viewpoint~\cite{sniedovich2006dijkstra}.

Lastly, we explore the {\em Alternating Direction of Multiplier Method} (ADMM)\cite{boyd2011distributed,boyd2004convex}, and a variant (InADMM) \cite{yue2018implementing}, for reducing the computational cost in identifying an approximate shortest path for very large graphs. The benefits of ADMM and its variants include: {\bf (i)} it admits distributed implementation, initialized with any suitable path ({\em warm-start}), if one is available. Such a feature speeds up convergence, is especially useful when a short path needs to be updated following topological changes in the graph. {\bf (ii)} it can be adopted as a solver for the parallel lasso \eqref{eq:parallel}, whereby multiple-pair and all-pair shortest path problems can be considered and formulated in the same manner. The ADMM algorithm proposed here is completely different than earlier proposals on the subject, e.g., the self-stabilizing approach in \cite{chandrasekar1994self}. The proposed algorithm aims at identifying an (approximate) shortest path between specified vertices, allowing for better computational complexity and relatively efficient updates to topological changes.

\blue
The present work builds on our earlier conference publication~\cite{dong2020lasso}, where we first presented the idea of relaxing the search for a shortest path via $\ell_1$-regularized regression, and pointed out analogies to the bi-directional Dijkstra algorithm. 
{\blue In~\cite{dong2020lasso}, in particular, we
sketched the LARS–Dijkstra connection and illustrated the idea via small synthetic examples. The present work aims to bring out the significance of the approach, sharpen the link between the two viewpoints, and provide detailed streamlined justification of points of contact. To this end, we herein provide detailed analysis of crossing/joining times (in the parlance of the LARS algorithm) and how/when these are manifested in the present setting (Section \ref{sec:LARS}), explore uniqueness of shortest path under both strong and weak regularity assumptions (Section \ref{sec:Lasso}), examine scalability, and discuss links to practical Dijkstra implementations from the new perspective (Section \ref{sec:proximal}). A focus in our exposition is on how to capitalize on fast numerical solvers to improve on speed and robustness. In particular, we detail insights in utilizing proximal lasso solvers (ADMM and an inexact variant based on sparse Cholesky and preconditioned conjugate gradients, including warm starts, over-relaxation, and stopping criteria). We report results on image grids, road networks, and random geometric graphs at scales from $10^3$ to $10^4$ nodes, with all parameters listed for reproducibility (Section \ref{sec:proximal}).}
\black

The outline of the work is as follows. Section~\ref{sec:prelim}  introduces notation along with basic concepts and a brief account of Dijkstra's algorithm. Section~\ref{sec:Lasso} casts the search for short paths in a network as a convex optimization problem as discussed. Section~\ref{sec:LARS} details the algorithmic steps for updating state-values on edges, that turn out to coincide with the so-called {\em lasso solution} in the LARS algorithm. Section~\ref{sec:relation}  highlights the commonality of features between the LARS algorithm and bi-directional Dijkstra algorithm. Finally, Section~\ref{sec:proximal} explores the application of the ADMM method to our lasso formulation, and highlights its relevance in identifying short, but not necessarily shortest, paths in very large graphs.

\section{Preliminaries } \label{sec:prelim} 
\subsection{Graph theoretic notations and definitions}
Throughout, we consider a weighted {\em undirected}  graph $\calG$ that is connected and has no self-loops or multi-edges.  We  write $\calG = (\calV, \calE, \calW)$,  where $\calV = \{v_1, \dots , v_n\}$ is the set of vertices/nodes, $\calE = \{e_{1}, \dots, e_{m}\}$ is the set of edges, and $\calW=\{w_1,\ldots,w_m\}$ a set of weights corresponding to the edges. 
We will consistently use $n=|\mathcal V|$ and $m=|\mathcal E|$ for the cardinality of these two sets. When labeling edges, we also use the notation $(v_i,v_j)$ for the edge that connects vertices $v_i$ and $v_j$, without significance to the order.

However, as is common, in defining the {\em incidence matrix} of the graph, denoted by $D(\calG)$ an arbitrary but fixed orientation is assigned to edges that has no bearing on the results.
To this end, the incidence matrix
is defined as the $n\times m$ matrix with $(i,j)$th entry
\begin{equation*}
[D]_{ij} =
\begin{cases}
+1 \ \   &\text{if the $i$th vertex is the tail of edge $e_{j}$},\\
-1  \ \   &\text{if the $i$th vertex is the head of edge $e_{j}$},\\
0  \ \   &\text{otherwise}.
\end{cases}  
\end{equation*}
It is convenient to define the {\em weight matrix}  $W=\text{diag}(w_1,\ldots,w_m)$ as the diagonal matrix formed by the weights in $\mathcal W$, consistent with the ordering in $\mathcal E$.

A path from vertex $v_s$ to vertex $v_t$ is a sequence of connected edges 
$$p=\{(v_{i_0},v_{i_1}),(v_{i_1},v_{i_2}),\ldots, (v_{i_{l-1}},v_{i_\ell})\}$$ 
that ``starts'' at $v_{i_0}=v_s$ and ``terminates'' at $v_{i_l}=v_t$.  An alternative representation of the path, which now encodes edge-orientation that is consistent with that in specifying $D$, is in terms of the {\em incidence vector} $\xp$. This is an $m$-dimensional vector defined as follows: 
the $i$th entry $(\xp)_i$ is $+1$, or $-1$, depending on whether an edge $(v_{i_{k-1}},v_{i_k})$ (for some $k\in\{1,\ldots,\ell\}$) in the path is the $i$th edge in $\mathcal E$ and is listed with orientation consistent or not with the tail/head designation in specifying $D$, respectively; if the $i$th edge is not in the path, $(\xp)_i=0$. The {\em length of the path} is defined as the sum of edge-weights, i.e.,
$$
\text{length}(p) \triangleq \sum_{e_i \in p} w_i = \|W\xp\|_1,
$$
where  $\|\cdot\|_1$ denotes the $\ell_1$-norm.

A \blue connected \black graph is said to be a {\em tree} if it has no cycle, i.e., it has no path where $s=t$. If $\mathcal G$ is connected, there is always a subgraph which is a tree. When $\mathcal G$ is a tree, $m=n-1$ and there is a unique path from any given vertex to any other.
Any vertex can be designated as {\em root}, and the structure of graph is encapsulated by all paths connecting vertices to the root ($n-1$ paths). The $(n-1)\times (n-1)$ matrix of incidence vectors of all such paths is referred to as {\em path matrix} (often with reference to the root) and denoted by $P_{v_1}$, or simply $P$, when the root is clear from the context. Interestingly, $P$ is closely connected to the incidence matrix $D$. This is the content of the following lemma, which is key for results in Section \ref{sec:relation} but also of independent interest.

\begin{lemma}\label{lem:incidence}
Let $\mathcal G$ be a tree rooted at $v_1$ with $n$ vertices and $P$ its path matrix. The pseudoinverse of its  incidence matrix $D$, denoted as $D^{+}$, is~\footnote{Throughout, $\mathbbm{1}_k$ denotes the $k$-column vector with entries equal to $1$, and $I_k$ the $k\times k$ identity matrix.}
\begin{equation*}
D^{+} =
\begin{bmatrix}
\displaystyle -\frac{1}{n}P\mathbbm{1}_{n-1}, & PJ
\end{bmatrix},
\end{equation*}
where $\displaystyle J=(I_{n-1}-\frac{1}{n} \mathbbm{1}_{n-1} \mathbbm{1}_{n-1}^T)$.
\end{lemma}

\begin{proof}
See \cite[Theorem 2.10 \& Lemma 2.15]{bapat2010graphs}.
\end{proof}

\subsection{Shortest path problem and Dijkstra's  algorithm}\label{sec:Dikstra}
Let $\calP_{s,t}$ denote the set of all paths between $v_s$ and $v_t$. This set is non-empty because the graph is connected. 
The shortest path problem is to find a path with minimum length over all the paths between $v_s$ and $v_t$, i.e.,
\begin{equation}\label{eq:shortest-path-problem}
{\rm arg}\min_{p\in \mathcal P_{s,t}}~\text{length}(p).
\end{equation}
The minimum value is known as the {\em distance} between $v_s$ and $v_t$. A well-known search algorithm - Dijkstra's algorithm, has been proposed for this problem.

Dijkstra's algorithm \cite{dijkstra1959note} begins with the ``starting'' vertex $v_s$, and initially assigns a distance of $0$ to $v_s$ and $+\infty$ to all other vertices. It iteratively labels the vertex with the lowest {\em distance estimate} as {\em visited} and updates the distance estimates of its neighbors that have not yet been visited ({\em unvisited}). The distance estimates are updated by summing up the distances of visited vertices and the weights of the edges linking these vertices to their unvisited neighbors. Dijkstra's algorithm terminates when $v_t$ is visited and produces the shortest path from source vertex $v_s $ to others vertices in the form of shortest-path tree.

The essential feature of Dijkstra’s algorithm is that it iteratively constructs the shortest-path tree rooted at $v_s$ to all the visited vertices before reaching the target $v_t$. Similarly, bi-directional Dijkstra algorithm constructs two shortest-path trees rooted at $v_s$ and $v_t$ and terminates when the two trees connect.

Later on in Section~\ref{sec:relation}, we will point out analogies between the bi-directional Dijkstra algorithm and properties of the LARS algorithm applied to the shortest-path problem advocated herein.

\section{Linear programming formulation}\label{sec:Lasso}
We now cast the shortest path problem~\eqref{eq:shortest-path-problem} as a linear program. To this end, we will use a well-known technique for finding sparse solutions to linear equations by minimizing the $\ell_1$ norm of a vector as a surrogate for the count of its non-vanishing entries \cite{tibshirani1996regression}.

In our setting, constraints are  expressed in terms of $D$ (incidence, also constraint matrix) and $x$, the incidence vector of a sought path $p$ from $v_s$ to $v_t$, in that,
\begin{equation}\label{eq:Dxy}
p \in \calP_{s,t} \quad \Rightarrow\quad D\xp  = \yst.
\end{equation}
Here, $\yst_i=\mathbf{1}_{\{s\}}-\mathbf{1}_{\{t\}}$ is the {\em indicator vector} in $\mathbb R^n$ of a {\em virtual edge} directly connecting $v_s$ to $v_t$; the path together with the virtual edge form a {\em closed cycle}. Throughout, $\mathbf{1}_{\{\cdot\}}$ denotes the indicator function of set $\{\cdot\}$.
Note that linear combinations $x=ax^{(p_1)} + (1-a)x^{(p_2)}$ with $a\in (0,1)$ of incidence vectors of two distinct paths $p_1$ and $p_2$ between $v_s$ and $v_t$, also satisfy the constraint $Dx=\yst$. 
That is, although an exact correspondence between the two sides of~\eqref{eq:Dxy} does not hold, it does hold between the {\em shortest path} and a corresponding integer vertex of the polytope defined by the constraint matrix $D$.
Thus, we {\em propose}
\begin{subequations}
\begin{align}\label{eq:linprog}
{\rm arg}\min_{x\in \mathbb R^m}~\|Wx\|_1,\quad \text{s.t.}\quad Dx = \yst
\end{align}
as a way to solve the shortest path problem.

To gain insight as to the nature of the minimizer,
problem \eqref{eq:linprog} can be recast as the linear program:
\begin{align}\label{eq:linprog1} 
{\rm arg}\min_{\xi\geq 0}~ \blue \sum_{i=1}^m w_i\,(x_i^+ + x_i^-),\black \quad \text{s.t.}\quad \mathcal D\xi  = \yst
\end{align}
\end{subequations}%
with 
$\mathcal D:= [D, -D]$ and $\xi\in\mathbb R^{2m}$. The solutions to \eqref{eq:linprog} and \eqref{eq:linprog1} correspond via 
\blue $\xi \;=\; \begin{bmatrix} x^+ \\ x^- \end{bmatrix} \ge 0$,
\black  where $x_+$ ($x_-$, resp.) is the vector of positive (negative, resp.) entries of $x$, setting zero for the negative (positive, resp.) entries, i.e., \blue 
$$
x^+ := \max(x,0),\qquad x^- := \max(-x,0),\qquad x = x^+ - x^-,\quad \mbox{and} \quad |x| = x^+ + x^-.
$$ \black 
The constraint matrix $\mathcal D$ in \eqref{eq:linprog1} is {\em totally unimodular} (i.e., all minors have determinant in $\{0,\pm 1\}$) \cite[Lemma 2.6]{bapat2010graphs}, and therefore, application of \cite[Theorem 11.11 (Unimodularity Theorem)]{ahuja1988network} \blue guarantees that there exists an integral optimal solution; if the shortest path is unique, then the optimal solution is unique and \black
$\{0,\pm 1\}$-valued. Thus, $x$ in \eqref{eq:linprog} corresponds to a valid incidence vector.

\subsection{Lasso formulation}
Returning to \eqref{eq:linprog}, rewritten in the form 
\begin{equation}
{\rm arg}\min_{\beta} \; \Big\{\|\beta\|_1\,\mid \, \|Q\beta - y\|^2_2=0 \Big \}   \tag{\ref{eq:linprog}$^\prime$}
\end{equation}
in new variables $\beta=Wx$, $y=\yst$, and $Q=DW^{-1}$,
leads us a relaxation as the $\ell_1$-regularized regression
\begin{equation}
{\boldsymbol{\beta}}(\lambda):={\rm arg}\min_{\beta \in \mathbb R^m}~ \frac{1}{2}\|y-Q\beta\|_2^2 + \lambda\|\beta\|_1, \label{eq:lasso}
\end{equation}
with (regularization parameter) $\lambda >0$. The formulation~\eqref{eq:lasso} is known as {\em lasso} \cite{tibshirani2013lasso}.

The limit ${\boldsymbol{\beta}}_0:=\lim_{\lambda\to 0}{\boldsymbol{\beta}}(\lambda)$ from~\eqref{eq:lasso}, for $\lambda>0$, provides the indicator vector $\boldsymbol{x}_0=W^{-1}{\boldsymbol{\beta}}_0$ of a path \blue under uniqueness and regularity (non-degeneracy) of the shortest path\black. 
\blue
Under these conditions, there exists $\bar\lambda>0$ such that $\mathrm{supp}\big(\boldsymbol{\beta}(\lambda)\big)=\mathrm{supp}\big(\boldsymbol{\beta}_0\big)$ for all $\lambda\in(0,\bar\lambda)$.
\black
For $\lambda>0$, ${\boldsymbol{\beta}}(\lambda)$ may not correspond to a path. However, due to continuity and the fact that $\boldsymbol{x}_0$ in the limit must be $\{0,\pm 1\}$-valued, for sufficiently small $\lambda$, ${\boldsymbol{x}(\lambda)}=W^{-1}{\boldsymbol{\beta}}(\lambda)$ reveals the shortest path (e.g., by rounding the values to the nearest integer). \blue In practice, applying a small threshold to $\boldsymbol{x}(\lambda)$ followed by a connectivity check reliably recovers the path. \black

Thus, the lasso formulation represents a viable and attractive approach for solving the shortest path problem. Interestingly, as we show in Section~\ref{sec:relation}, the LARS algorithm -- a popular solver for lasso~\eqref{eq:lasso}, shares features of the bi-directional Dijkstra algorithm. Most importantly, the lasso formulation~\eqref{eq:lasso}, as discussed in Section~\ref{sec:proximal}, allows the use of proximal optimization methods for obtaining satisfactory approximations of the shortest path in large graph settings.

\subsection{Uniqueness of the  lasso solution}
A sufficient condition for uniqueness of solution to \eqref{eq:lasso} is that $\text{rank}(Q)=m$, the size of $\beta$ and number of edges ~\cite[Lemma 2]{tibshirani2013lasso}.
However, recall that $WQ=D$, the incidence matrix. It follows that $\text{rank}(Q)=m$ only holds when the graph is a tree (or possibly, a disjoined set of trees, cf.~\cite[Theorem 2.3]{bapat2010graphs}). 
Evidently, such an assumption is too restrictive, also since the shortest path problem in this case becomes trivial.

Herein, we introduce a fairly general sufficient condition for the uniqueness of solution to \eqref{eq:lasso} (as we claim next), that is in fact generic, for generic weights.

\begin{assumption}[\bf Strong (terminal–to–all) uniqueness]\label{assum:uni}
The shortest path between vertex $v_s$ and any other vertex is unique, and the same applies to $v_t$. 
\end{assumption}

We note that ${\boldsymbol{\beta}}(\lambda)$ from \eqref{eq:lasso} turns out to be piece-wise linear (see Section \ref{sec:LARS}). The values of $\lambda$ where the slope changes are referred to as {\em breakpoints}. With this in place, the implications of the assumption to our problem can be stated as follows.

\begin{lemma}[\bf Uniqueness]\label{lemma:uni}
Under Assumption~\ref{assum:uni}, Problem~\eqref{eq:lasso} admits a unique solution for all $\lambda > 0$.
\end{lemma}

It is important to note, \blue in general, that the quadratic expression in \eqref{eq:lasso} may not be \black strictly convex \blue --  when $\mathrm{rank}(Q)\le n-1<m$, $Q$ has a nontrivial null space, whereas on trees where  $\mathrm{rank}(Q)=m$, strict convexity is ensured. \black The key idea in proving uniqueness under the conditions of the lemma requires a discussion of the LARS algorithm that is explained next. Hence the proof is deferred to Appendix~ \ref{apdx:proof-lemma-3.2}.

\blue
\begin{assumption}[\bf Weak ($s$–$t$) uniqueness]\label{assum:st-uni}
The shortest $s$--$t$ path is unique, whereas shortest paths from $s$ (or from $t$) to other vertices may be non-unique.
\end{assumption}

\begin{remark}[\bf Consequences under Assumption~\ref{assum:st-uni}]
Under Assumption~\ref{assum:st-uni}, problem in~\eqref{eq:lasso} has a minimizer for every $\lambda>0$. For almost all $\lambda$ the minimizer is unique. Non-uniqueness can occur only at parameter values where ties create multiple geodesics from $s$ or $t$ to intermediate vertices. At such values, LARS may admit simultaneous edge arrivals and the active set can exhibit crossings.
\end{remark}
\black

\section{The LARS algorithm}\label{sec:LARS}
\subsection{Karush-Kuhn-Tucker (KKT) conditions}
The solution $\bbeta(\lambda)$ of~\eqref{eq:lasso} must satisfy the  {\em KKT condition}{ ~\cite[Section 2.1]{tibshirani2013lasso}}
\begin{align}\label{eq:KKT}
Q^T\Big(y - Q\beta(\lambda)\Big) = \lambda \gamma,
\end{align}
where the vector $\gamma$ is in the sub-differential of $\|\beta(\lambda)\|_1$, with $j$th component given by
\begin{align*}
\gamma_{j} \in
\begin{cases}
\Big\{\text{sign}\big(\beta_{j}(\lambda)\big)\Big\}~~  &\text{if}\;\beta_{j}(\lambda) \neq 0,\\[1em]
\left[ -1, 1 \right] ~~ &\text{if}\; \beta_{j}(\lambda)=0.
\end{cases}
\end{align*}
\noindent The KKT condition in \eqref{eq:KKT} motivates us to divide the indices $\{1,\ 2,\dots, m\}$ into two set: an {\em active} set (also, {\em equicorrelation} set \cite{tibshirani2013lasso})  $\calA$ and, a {\em non-active} set $\calA^c$, i.e.,
\begin{align}\label{eq:active-set-def}
\calA \triangleq \Big\{j \ | \ \bbeta_{j}(\lambda) \neq 0\Big\}, \qquad \calA^{c} \triangleq \Big\{j\ |\ \bbeta_{j}(\lambda) = 0\Big\}.
\end{align}
Let $\bbeta_\calA(\lambda)$ denote the vector $\bbeta(\lambda)$ with the non-active (or equivalently, zero) entries removed and, likewise, $Q_\calA$ be the matrix $Q$ with the columns corresponding to the non-active set removed,  and thus fully \blue depends \black on $\bbeta(\lambda)$. Then, the KKT condition~\eqref{eq:KKT} can be expressed as
\begin{subequations}\label{eq:KKT-new}
\begin{align}\label{eq:KKT-active}
&Q_j^T\Big(y-Q_\calA\bbeta_\calA(\lambda)\Big) = s_j \lambda,\quad \forall j \in \calA,\\\label{eq:KKT-non-active}
&\Big|Q_j^T\Big(y-Q_\calA\bbeta_\calA(\lambda)\Big)\Big| \leq \lambda,\phantom{xx} \forall j \in \calA^c, 
\end{align}
\end{subequations}
where $Q_j$ denotes the $j$th column of $Q$ and the {\em sign vector}
\begin{align}\label{eq:sign-vec}
s := \text{sign}\left(Q_\calA^{T}\Big(y - Q_\calA\bbeta_\calA(\lambda )\Big) \right) = \text{sign}\left(\bbeta_\calA \right),    
\end{align}
with $j$th element the sign ($\pm 1$) of the $j$th entry of $\bbeta_{\mathcal A}$.

\subsection{The LARS algorithm}\label{sec:lars-alg}
The LARS algorithm, as formulated in ~\cite[Section 3.1]{tibshirani2013lasso} to solve lasso, finds the {\em solution path} of $\bbeta(\lambda)$ that meets the KKT condition~\eqref{eq:KKT} for all $\lambda>0$. The algorithm is initialized with $\lambda_0=\infty$, $\calA_0=\emptyset$, and $s_0=\emptyset$. The solution path $\bbeta_{\calA}(\lambda)$, in \eqref{eq:betaA} below, is computed for decreasing $\lambda$, and is piece-wise linear and continuous with breakpoints $\lambda_0>\lambda_1>\ldots>0$. Breakpoints are successively computed at each iteration of the algorithm, and the linear segment of $\bbeta_{j}(\lambda)$ is determined to satisfy the element-wise KKT condition as detailed in~\eqref{eq:KKT-new}. As $\lambda$ crosses breakpoints, the active (non-active) set \eqref{eq:active-set-def} and the sign vector \eqref{eq:sign-vec} are updated accordingly.

We now detail the $k$th iteration of the LARS algorithm, initializing $\lambda = \lambda_k$, $\calA = \calA_k$, $s = s_k$, and seeking the next breakpoint $\lambda_{k+1}$. The lasso variable $\bbeta_{\calA_{k}}(\lambda)$, as a function of $\lambda$, is calculated as the minimum $\ell_2$-norm solution of~\eqref{eq:KKT-active}, and is given by:
\begin{align}\label{eq:betaA}
\bbeta_{\calA_k}(\lambda)\!=\!
\Big(Q_{\calA_k}^{T}Q_{\calA_k}\Big)^{+}\Big(Q_{\calA_k}^{T}y-\!\lambda s_k\Big)
\!= \!a^{(k)} - b^{(k)} \lambda,
\end{align} 
where $(\cdot)^+$ denotes the Moore–Penrose pseudoinverse\footnote{\blue Herein, $W$ is positive diagonal and $D$ has full column rank (the incidence on a tree). Then we have $(DW^{-1})^{+}=(W^{-1})^{-1}D^{+}=W\,D^{+}$.} and  
\begin{align}\label{eq:a-b}
\begin{aligned}
a^{(k)} &:= \Big(Q_{\calA_k}^{T}Q_{\calA_k}\Big)^{+}Q_{\calA_k}^{T}y,\\ 
b^{(k)} &:= \Big(Q_{\calA_k}^{T}Q_{\calA_k}\Big)^{+}s_k.   
\end{aligned}
\end{align}

The next breakpoint $\lambda_{k+1}$ is determined as the largest value at which 
the KKT condition $\big(Q_j^T(y - Q_{\calA_k}\beta_{\calA_k}(\lambda)) = \lambda\gamma_{j}\big)$ is violated by $\lambda$. Such a violation occurs in two circumstances (``crossing/joining'' in the language of \cite{tibshirani2013lasso}), either \eqref{eq:KKT-active} or \eqref{eq:KKT-non-active} fails.

{\noindent \bf {\em i})} The ``joining'' case is when condition~\eqref{eq:KKT-non-active} is violated for some $j\in \calA^c_k$, i.e., 
$$\Big|Q_j^T\Big(y-Q_{\calA_k} \bbeta_{\calA_k}(\lambda)\Big)\Big| \leq \lambda
$$ 
no longer holds. 
For each index $j\in \calA^c_k$, this happens at $\lambda= t^{\jn}_j$ given by 
\begin{equation}\label{eq:join}
t^{\jn}_{j,k} = \frac{Q_j^T\Big(Q_{\calA_k}a^{(k)}-y\Big)}{Q_j^TQ_{\calA_k}b^{(k)} \pm 1},
\end{equation}
where the choice $\pm$ is the one for which $t^{\jn}_{j,k}\in[0,\lambda_{k}]$.
We set ``joining time'' 
$$
\lambda_{k+1}^{\jn} := \max_{j \in \calA_k^c} \; \Big\{t^{\jn}_{j,k}\Big\}.
$$

{\noindent \bf {\em ii})} The ``crossing'' case is when condition~\eqref{eq:KKT-active} is violated for some $j\in \calA_k$  so that \blue one of the elements \black of $\bbeta_{\calA_k}(\lambda)$ crosses zero (changes its sign), i.e. $a^{(k)}_j - \lambda b^{(k)}_j=0$ for some $\lambda<\lambda_k$. 
For each index $j\in \calA_k$, the crossing happens at $\lambda=t^{\cs}_j$ given by  
\begin{equation}\label{eq:cross}
t_{j,k}^{\cs} = 
(a^{(k)}_j/b^{(k)}_j) \cdot \mathbf{1}_{\{0< (a^{(k)}_j/b^{(k)}_j) < \lambda_k\}}.
\end{equation}
We set ``crossing time'' 
$$
\lambda_{k+1}^{\cs} = \max_{j \in \calA_k} \Big\{t_{j,k}^{\cs}\Big\}.
$$
The next breakpoint is
\begin{equation}\label{eq:next-break}
\lambda_{k+1}=\max \Big\{\lambda_{k+1}^{\jn}, \lambda_{k+1}^{\cs}\Big\}.
\end{equation}
If the joining occurs, the joining index is added to the active set, and the sign vector is updated. If a crossing happens, the crossing index is removed from the active set. The overall algorithm is \blue summarized as follows \black.
\begin{algorithm}[htb!]
\caption{\bf (LARS algorithm for the lasso ~\eqref{eq:lasso})}
\begin{algorithmic}[1]
\renewcommand{\algorithmicrequire}{\textbf{Input:}}
\renewcommand{\algorithmicensure}{\textbf{Output:}}
\REQUIRE matrix $Q=DW^{-1}$, vector $y=\yst$.
\ENSURE  incidence vector $\boldsymbol{x}_{0}\!=\!W^{-1}\bbeta_{0}$, distance $\|\bbeta_{0}\|_1$.
\STATE  $k=0$, $\lambda_{0}=\infty$, $\calA=\emptyset$, $s=0$, $a^{(0)} =0$ and $b^{(0)}=0$. 
\WHILE{$\lambda_{k}>0$} 
\STATE Compute the joining time $\lambda_{k+1}^{\jn}$~\eqref{eq:join} for $j\in \calA^c_k$.
\STATE Compute the crossing time $\lambda_{k+1}^{\cs}$~\eqref{eq:cross} for $j\in \calA_k$.
\STATE Compute $\lambda_{k+1}$ according to~\eqref{eq:next-break} and
\begin{enumerate}[label=\roman*)]
  \item if join happens, i.e., $\lambda_{k+1}=\lambda^{\jn}_{k+1}$, add the joining index $j$ to $\calA_k$ and its sign to $s_k$;
  \item if cross happens, i.e., $\lambda_{k+1}=\lambda^{\cs}_{k+1}$, remove the crossing index $j$ from $\calA_k$ and its sign from $s_k$.
\end{enumerate}
\STATE $k = k+1$.
\STATE Update $Q_{\calA_{k}}$. according to $\calA_{k}$ and $\calA^{c}_{k}$.
\STATE Compute $a^{(k)}$ and $b^{(k)}$ according to~\eqref{eq:a-b}.
\STATE Set $\bbeta_{\calA_k}  = a^{(k)} - \lambda_kb^{(k)}$ and $\bbeta_{\calA_k^c}=0$.
\ENDWHILE
\end{algorithmic}
\label{alg:LARS}
\end{algorithm}

\subsection{Numerical example}
We consider Nicholson's graph \cite[p.~6]{pohl1969bidirectional} and seek the shortest path between vertex $v_1$ and $v_9$. The iterations of the LARS algorithm (with breakpoints $\lambda_0=+\infty$, $\lambda_1 = 1/2$, $\lambda_2=1/3$, $\lambda_3 = 1/5$, $\lambda_4 = 0.1489$) are depicted in Fig.~\ref{fig:Nicholson}. It is observed that, at each iteration, edges highlighted in the same color as the corresponding $\lambda$ are added to the active set and are never removed prior to the last step $\lambda_{5} = 0$, as we prove later on in Lemma~\ref{lemma:cross}. The algorithm terminates after four iterations when $\lambda_5=0$ and a path between vertex $v_1$ and $v_9$ is formed. The element-wise path of the lasso solution $\bbeta(\lambda)$ as $\lambda$ decreases, is drawn in Fig.~\ref{fig:lassopath}.

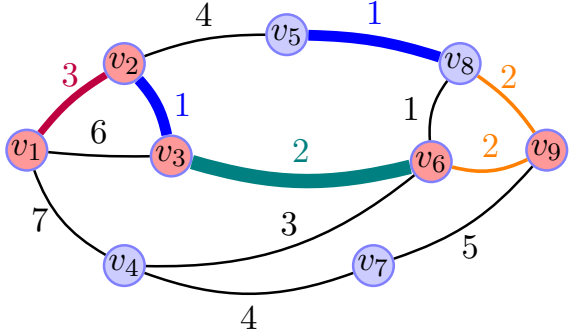
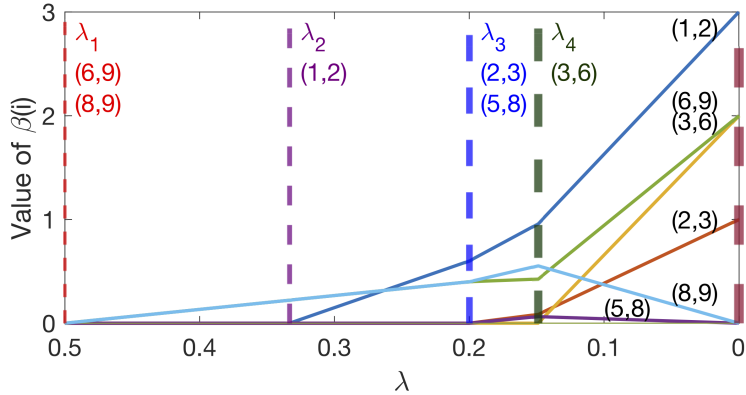
\begin{figure}[htb!]
\begin{subfigure}{1\columnwidth}
\centering
\scalebox{1.2}{
\begin{tikzpicture}
[process/.style={circle,draw=blue!50,fill=blue!20,thick,inner sep=0pt,minimum size=4.5mm},
target/.style={circle,draw=blue!50,fill=red!40,thick,inner sep=0pt,minimum size=4.5mm}]
\node (n1) at (0,0) [target] {$v_1$};
\node (n2) at (1.08,0.96) [target] {$v_2$};
\node (n3) at (1.6,-0.065) [target] {$v_3$};
\node (n4) at (1.08,-1.28) [process] {$v_4$};			
\node (n5) at (2.88,1.28) [process] {$v_5$};
\node (n6) at (4.48,-0.128) [target] {$v_6$};
\node (n7) at (3.84,-1.28) [process] {$v_7$};
\node (n8) at (4.8,0.96) [process] {$v_8$};
\node (n9) at (5.76,0) [target] {$v_9$};  
				
\draw[line width=.8mm][purple] (n1) to [bend left=10] node[midway,above](){3} (n2);
\draw[thick] (n1) to [bend right=3] node[midway,above](){6} (n3);
\draw[thick] (n1) to [bend right=20] node[midway,left](){7} (n4);
\draw[thick] (n2) to [bend left=10] node[midway,above](){4} (n5);
\draw[line width=1.2mm][blue] (n2) to [bend left=15] node[midway,right](){1} (n3);
\draw[line width=1.6mm][teal] (n3) to [bend right=15] node[midway,above](){2} (n6);
\draw[thick] (n4) to [bend right=20] node[midway,above](){3} (n6);
\draw[thick] (n4) to [bend right=20] node[midway,below](){4} (n7);
\draw[line width=1.2mm][blue] (n5) to [bend left=8] node[midway,above](){1} (n8);
\draw[thick] (n6) to [bend left=20] node[midway,left](){1} (n8);
\draw[line width=0.4mm][orange] (n6) to [bend right=20] node[midway,above](){2} (n9);
\draw[thick] (n7) to [bend right=15] node[midway,below](){5} (n9);
\draw[line width=0.4mm][orange] (n8) to [bend left=10] node[midway,above](){2} (n9);
\end{tikzpicture}}
\begin{minipage}{0.95\columnwidth}
\caption{Nicholson's graph: Edges are added according to breakpoints $\lambda$ in the same color as in Fig.~\ref{fig:lassopath}.} 
\label{fig:Nicholson}
\end{minipage}
\end{subfigure}

\centering
\begin{subfigure}{1\columnwidth}
\centering
\includegraphics[width=0.7\columnwidth]{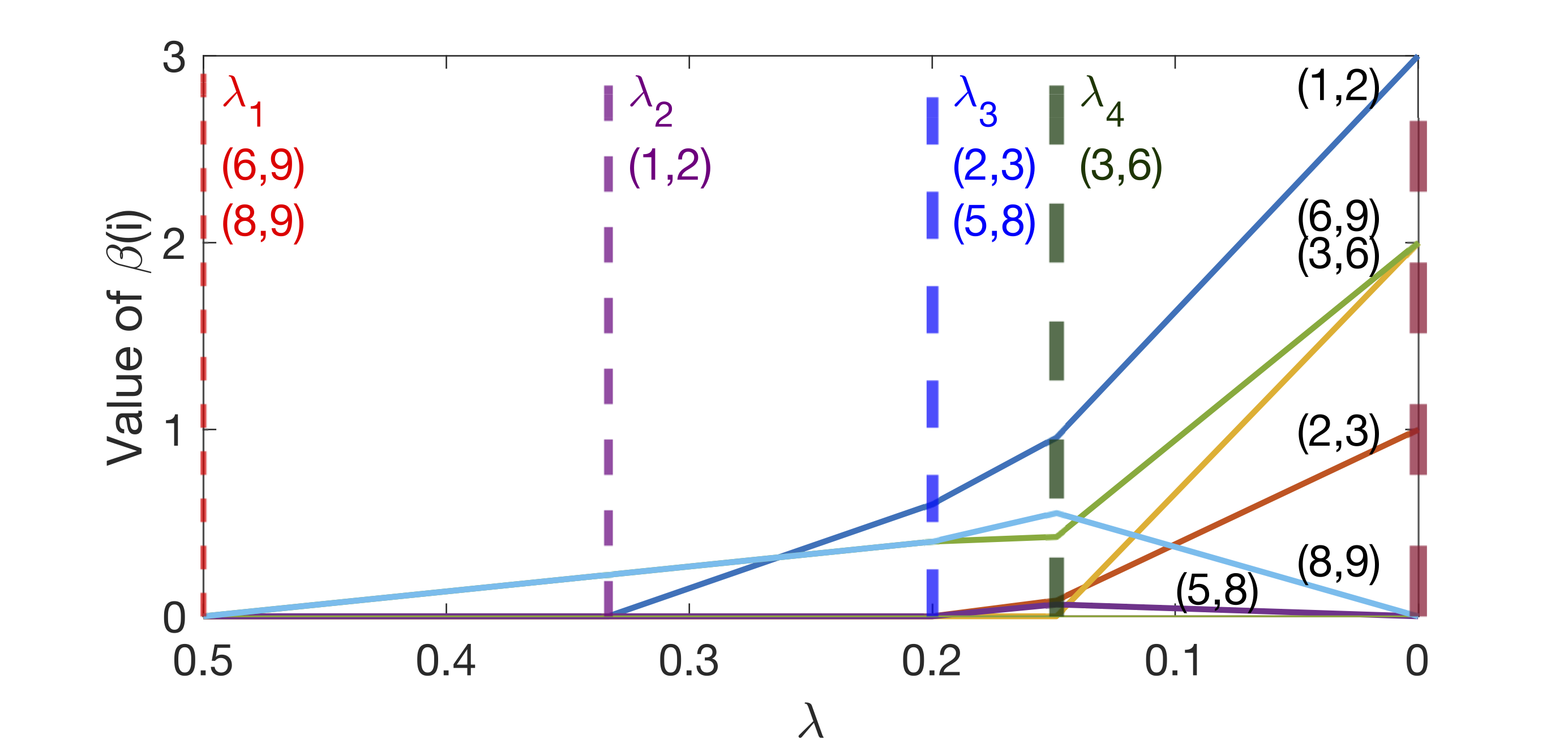}
\begin{minipage}{1\columnwidth}
\caption{Entries $(\bbeta(\lambda))(i)$ (corresponding to edges) \blue drawn \black as functions of $\lambda$. Values of $\lambda$, where the active edge-set changes, are marked with dashed vertical lines.}
\label{fig:lassopath}
\end{minipage}
\end{subfigure}  
\caption{The LARS algorithm successively identifies a set of active edges while reducing the tuning/control parameter $\lambda$. A vector $\beta(\lambda)$ with information of the length-contribution of the active edge-set is also successively being updated.} 	
\label{fig:example}
\end{figure}

Example~\ref{fig:example} highlights the correlation between the LARS algorithm and Dijkstra's algorithm. Specifically, the LARS algorithm builds two shortest-path trees, with roots at vertices $v_1$ and $v_9$. This echoes the steps in the bi-directional Dijkstra algorithm discussed in Section~\ref{sec:Dikstra}. In Theorem~\ref{thm:thm1}, we highlight similarities between the LARS and Dijkstra's algorithms.

\section{Relationship between lasso and Dijkstra}\label{sec:relation}
Both the LARS as well as the bi-directional Dijkstra algorithm iteratively construct shortest-path trees with roots at $v_s$ and $v_t$, and terminate when the two trees meet. We prove our main result below by induction. The induction hypothesis is specified next.
\begin{assumption}[\bf Induction hypothesis]\label{assum:hypothesis}
At $(k-1)$th iteration of LARS, the edges in the active set $\calA_{k-1} = \calA_{k-1}^{(s)} \cup \calA_{k-1}^{(t)}$, where $\calA_{k-1}^{(s)}$ and $\calA_{k-1}^{(t)}$ are disjoint subsets of edges forming trees on vertices $T^{(s)}_{k-1}, T^{(t)}_{k-1}\subset\calV$, rooted at $v_s$ and $v_t$, respectively. The two trees are the shortest-path trees from the roots. Moreover, crossing does not occur at this iteration, i.e., no edges are removed from the active set.
\end{assumption}

Our induction starts with the {\em base case} $k=1$, where the active set is empty and both trees consist of a single root vertex, $T^{(s)}_0=\{s\}$ and $T^{(t)}_0=\{t\}$. The crossing does not occur since the active set is empty. The proof is based on the following two propositions that provide simplified expressions for the joining and crossing times, derived from Lemma~\ref{lem:incidence} and Assumption~\ref{assum:hypothesis} that the graph is a tree. The proofs are given in Appendices~\ref{apdx:t-cross-derivation} and~\ref{apdx:t-join-derivation}. 
\blue 
Throughout this section and throughout the Appendices, we use $D$, $D^{+}$, $Q$, and $Q^{+}$ to mean $D_{\mathcal A_k}$, $D_{\mathcal A_k}^{+}$, $Q_{\mathcal A_k}$, and $Q_{\mathcal A_k}^{+}$ on the active edge set $\mathcal A_k$ of the current trees $T_k^{(s)}$ and $T_k^{(t)}$, with $Q := D W^{-1}$. Since each active component is a tree, the incidence has full column rank. Hence, for any tree $T$, we have $(D_T W_T^{-1})^{+}=W_T D_T^{+},$
where $W_T$ is diagonal with positive entries.
\black

\subsection{{  Joining and crossing times}}
\begin{proposition}[{\bf Joining time}]\label{prop:joining}
For the edge $e_j=(v_1,v_2)$ where $e_j\in \calA_k^c$, the {  element-wise} joining time is  
\begin{align*} 
t^{\jn}_{j,k} = 
\begin{cases}
\begin{aligned}
0 \hspace{1.45in} \text{if}~(v_1,v_2) \in \Omega_k^2  \cup  {T^{(s)}_k}^2 \cup {T^{(t)}_k}^2 &\\[2em]
\Big(|T^{(s)}_k|l^{(s)}_{v_2} \! - \!\! \displaystyle \sum_{v\in T^{(s)}_k} l^{(s)}_v\Big)^{-1} \hspace{0.65in} \text{if}~(v_1,v_2) \in {T^{(s)}_k} \times \Omega_k &\\[1em]
\Big(|T^{(t)}_k|l^{(t)}_{v_2} \! - \!\! \displaystyle\sum_{v\in T^{(t)}_k} l^{(t)}_v\Big)^{-1} \hspace{0.7in} \text{if}~(v_1,v_2) \in {T^{(t)}_k} \times \Omega_k &\\[1em]
\Big(|T^{(s)}_k|+|T^{(t)}_k|\Big)/ \blue \Delta \black \hspace{0.9in} \text{if}~(v_1,v_2) \in {T^{(s)}_k} \times{T^{(t)}_k} &
\end{aligned}
\end{cases}
\end{align*}
where $\Omega_k= \calV \setminus (T^{(s)}_k \cup T^{(t)}_k)$, $l^{(s)}_v$ and $l^{(t)}_v$ denote the distance of vertex $v$ to the root $s$ and $t$ respectively, and 
$$
\blue \Delta := |\Ts_k| |\Tt_k| l^{(s)}_t - |\Tt_k| \sum_{v\in \Ts_k}l_v^{(s)} - |\Ts_k|\sum_{v\in \Tt_k}l_v^{(t)}. \black
$$
\end{proposition}

\begin{proposition}[\bf Crossing time]\label{prop:crossing}
For an edge $e_j=(v_1,v_2)$ where $e_j\in \calA_k$, the expression $a_{j}^{(k)}/b_j^{(k)}$ that appears in the definition of crossing time~\eqref{eq:cross} is
\begin{align*}
\frac{a_j^{(k)}}{b_j^{(k)}}\!=\!\!
\begin{cases}
\Big((|\Ts_k|/|R^{(s)}_j|)\!\!\!\!\sum\limits_{v \in R^{(s)}_j} \!\! l^{(s)}_v \!\! - \!\!\!\!\!\sum\limits _{v\in T^{(s)}_k} \!\!l^{(s)}_v\Big)^{-1},\qquad\text{if}~(v_1,v_2)\in {T^{(s)}_k}^2 \\[2em]
\Big((|\Tt_k|/|R^{(t)}_j|)\!\!\!\!\sum\limits _{v \in R^{(t)}_j} \!\! l^{(t)}_v \!- \!\!\!\!\!\sum\limits _{v\in T^{(t)}_k} l^{(t)}_v\Big)^{-1},\qquad\text{if}~(v_1,v_2)\in {T^{(t)}_k}^2 
\end{cases}
\end{align*}
where $R^{(s)}_j$ and $R^{(t)}_j$  are the subsets of vertices in the tree $T^{(s)}_k$ and $T^{(t)}_k$  respectively, whose path to the root contains the edge $e_j$. 
\end{proposition}

\subsection{Main result}
Assuming the induction hypothesis in Assumption~\ref{assum:hypothesis} holds, we show that the hypothesis also holds at iteration $k$ through the following lemmas.
\begin{lemma}[\bf Edge adding]\label{lemma:addactive} At iteration $k$,
either the edge connecting $v_\text{min}^{(s)}$ to tree $T^{(s)}_k$, or the edge $v_\text{min}^{(t)}$ to tree $T^{(t)}_k$ will be added to the active set, where $v_\text{min}^{(s)}$ and $v_\text{min}^{(t)}$ are vertices with minimum distance to roots $v_s$ and $v_t$ among all other vertices outside the two trees, respectively.
\end{lemma}

\begin{proof}
First, assume no edge connects the two trees, i.e., the last case of joining time does not happen (the case when it does is studied in Lemma~\ref{lemma:terminate}). The joining time $\displaystyle \lambda^{\jn}_{j,k}$ now reads
\begin{equation*}
\max \bigg\{\Big(|\Ts_k|l_{v^{(s)}_{\min}}  \!-\!\sum\limits _{v\in T^{(s)}_k} l^{(s)}_v \Big)^{-1}\!\!\!\!\!,\;\;\; {\blue \Big(|\Tt_k|l_{v^{(t)}_{\min}} \! - \! \sum\limits _{v\in T^{(t)}_k} l^{(t)}_v\Big)^{-1} \bigg\}},
\end{equation*}
where the first expression is achieved by the edge that connects $v_\text{min}^{(s)}$ to tree $T^{(s)}_k$ and the second is achieved by the edge that connects $v_\text{min}^{(t)}$ to tree $T^{(t)}_k$. Hence, one of these two edges is joined to the active set, if crossing does not occur.
\end{proof}

\begin{remark}[\bf No cycles] Cycles may be created in the following two scenarios: {$(i)$} An edge that connects two vertices of a tree is joined; $(ii)$ Two edges that connect the tree to a single vertex, say $v$, are joined simultaneously. Scenario $(i)$ can not happen because $t_j^{\jn}=0$ for such edges (first case of joining time). Scenario $(ii)$ can not happen, because in order for two edges to join simultaneously, we must have two distinct shortest paths from $v$ to the root, which is not possible according to Assumption \ref{assum:uni}. 
\end{remark}

\begin{lemma}[\bf No edge removed]\label{lemma:cross}
At iteration $k$ and if $\lambda>0$, crossing does not take place.
\end{lemma}

\begin{proof}
To prove that crossing does not take place before the algorithm terminates, we show that $a^{(k)}_j/b^{(k)}_j\geq \lambda_{k-1}$ and hence $a^{(k)}_j/b^{(k)}_j\geq \lambda_{k}$ for all $e_j$ in the active set, so that crossing time is zero according to its definition~\eqref{eq:cross}.

First, we obtain an expression for $\lambda_k$ and then compare it to crossing times. The value $\lambda_k$ is determined by the maximum of joining time and crossing time at $(k-1$)th iteration according to~\eqref{eq:next-break}. Under Assumption~\ref{assum:hypothesis}, the breakpoint $\lambda_k$ is the maximum joining time, which takes two possible values, corresponding to the edge that connects to either tree $T^{(s)}_{k-1}$ or tree $T^{(t)}_{k-1}$ as proved in Lemma~\ref{lemma:addactive}. 

Without loss of generality, we now assume the joining happens to the tree $T^{(s)}_{k-1}$. Then, 
\begin{align}
\lambda_{k-1}= \bigg(|\Ts_{k-1}|l_{v^{(s)}_{\min}}  -\sum_{v\in T^{(s)}_{k-1}} l^{(s)}_v\bigg)^{-1}.
\end{align}
Next, we show $a^{(k)}_j/b^{(k)}_j\geq \lambda_{k-1}$ for all $e_j$ that belong to the tree $T^{(s)}_k$. {  From Proposition~\ref{prop:crossing},  we can write $(a_j^{(k)}/b_j^{(k)})$ as} 
\begin{align*}
\frac{a_j^{(k)}}{b_j^{(k)}}
&=\bigg(\frac{1+|T^{(s)}_{k-1}|}{|R^{(s)}_j|} \sum_{v \in R^{(s)}_j} l^{(s)}_v - l_{v^{(s)}_{\min}} - \sum_{v\in T^{(s)}_{k-1}} l^{(s)}_v\bigg)^{-1}\\
&\geq \bigg(|T^{(s)}_{k-1}| l_{v^{(s)}_{\min}} - \sum_{v\in T^{(s)}_{k-1}} l^{(s)}_v\bigg)^{-1},
\end{align*}
where we used 
\begin{equation*}
|T^{(s)}_k|=|T^{(s)}_{k-1}|+1,\quad \sum_{v\in T^{(s)}_{k}} l^{(s)}_v=  l_{v^{(s)}_{\min}} + \sum_{v\in T^{(s)}_{k-1}} l^{(s)}_v, \quad \mbox{ and } \;\; l_{v^{(s)}_{\min}}  \geq l_{v}
\end{equation*}
for all $v\in T^{(s)}_k$. The inequality above holds because $v^{(s)}_{\min}$ is the latest vertex that is added to the tree, and other vertices that have already been added have a shorter distance to the root.

The proof of $a^{(k)}_j/b^{(k)}_j\geq \lambda_{k-1}$ for all $e_j$ that belong to the other tree $T^{(t)}_k$ is conceptually similar. One needs to compare $a^{(k)}_j/b^{(k)}_j$ with the joining time of the last edge that has been added to the tree $T^{(t)}_k$ at a certain past iteration, say $k'<k$, and use the fact that $\lambda_k<\lambda_{k'}$. The details are omitted on account of space. 
\end{proof}

\blue Recall from Section~\ref{sec:lars-alg}, that the LARS path yields a strictly decreasing sequence of breakpoints $\lambda_0>\lambda_1>\cdots>\lambda_K=0$. Hence, once the two trees connect at iteration $k$, it suffices to verify that all remaining candidate joining and crossing times are zero, which forces $\lambda_{k+1}=0$. \black

\begin{lemma}[\bf Termination criteria]\label{lemma:terminate}
The LARS algorithm terminates when the two trees connect.
\end{lemma}

\begin{proof}
Assume the two trees $\Ts_k$ and $\Tt_k$ become connected at iteration $k$. This happens when the last expression of joining achieves the maximum joining time, hence {\blue $\lambda_k = \big(|\Ts_k| + |\Tt_k|\big) / \Delta$}, as depicted in Fig.~\ref{fig:part-c}. The objective is to show that the algorithm terminates after this, i.e. $\lambda_{k+1}=0$. We show this by proving that the joining time and crossing time are both zero.

The derivation of element-wise joining time in  Proposition~\ref{prop:joining} reveals that $t_{j,k+1}^{join} = 0,~\forall e_j \in \calA^{c}_{k+1}$. 
For the crossing time, the derivation of   Proposition~\ref{prop:crossing} yields that for all $e_j \in \calA_{k+1}$,
\begin{align*}
\frac{a_j^{(k+1)}}{b_j^{(k+1)}}=
\begin{cases}
0, &\text{if}~e_j \notin p_{s,t}\\[0.5em]
\Big(\displaystyle\sum_{v\in R_j} l_v^{(s)} - \frac{|R_j|}{|\Tt_k|+|\Ts_k|} \displaystyle\sum_{v\in\Ts_k \cup \Tt_k} l_v^{(s)}\Big)^{-1}, \quad &\mbox{otherwise}
\end{cases}
\end{align*}
where $p_{s,t} \subset \calA_{k+1}$ is the path from $v_s$ to $v_t$. Therefore, it remains to show that the crossing time for the edges in $\calA_{k+1}\cap p_{s,t}$ are zero. We show this by proving $a_j^{(k+1)}/b_j^{(k+1)}\geq \lambda _k$ for all edges $e_j\in \calA_{k+1}\cap p_{s,t}$. Considering the edges that belong to the tree $\calA^{(s)}_{k}$, we have 
\begingroup
\allowdisplaybreaks
\begin{align*}
\frac{|\Ts_k|+|\Tt_k|}{\Big(a_j^{(k+1)}/b_j^{(k+1)}\Big)} = \blue \Delta \black + |\Ts_k| \sum_{v\in R_j^{(s)}}l_v^{(s)} - |R_j^{(s)}|\sum_{v \in \Ts_k} l_v^{(s)}
- |\Tt_k| \sum_{v\in R_j^{(s)}} l_v^{(t)} + |R_j^{(s)}|\sum_{v \in \Tt_k} l_v^{(t)},
\end{align*}
\endgroup
where $R_j^{(s)}$ are the vertices in the tree $\Ts_k$ such that their path to the root $v_s$ contains $e_j$.  Because $l_v^{(s)} \leq l_{v_2}^{(s)}$ and $l_v^{(t)} \geq l_{v_2}^{(t)}$  for all $v\in R_j^{(s)}$, we have the inequality
\begingroup
\allowdisplaybreaks
\begin{align*}
\frac{|\Ts_k|+|\Tt_k|}{\Big(a_j^{(k+1)}/b_j^{(k+1)}\Big)} - \blue \Delta \black \leq~|R_j|\bigg(|\Ts_k| l_{v_2}^{(s)}-\sum_{v \in \Ts_k} l_v^{(s)}
 - |\Tt_k|  l_{v_2}^{(t)} + \sum_{v \in \Tt_k} l_v^{(t)} \bigg).
\end{align*}
\endgroup
We claim that the expression in parentheses is negative
\begin{equation}\label{eq:claim}
|\Ts_k| l_{v_2}^{(s)}-\sum_{v \in \Ts_k} l_v^{(s)}  - |\Tt_k|  l_{v_2}^{(t)} + \sum_{v \in \Tt_k} l_v^{(t)} \leq 0, 
\end{equation} 
and if the claim is true, we have
$$
a_j^{(k+1)}/b_j^{(k+1)} \geq \bigg(|\Ts_k| + |\Tt_k|\bigg)/\blue \Delta \black = \lambda_k,
$$
so that the crossing time is zero for edges $e_j \in \calA^{(s)}_k$.
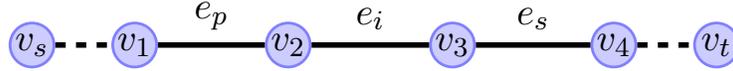
\begin{figure}[htb!]
\centering
\scalebox{1.3}{
\begin{tikzpicture}
[process/.style={circle,draw=blue!50,fill=blue!20,thick,inner sep=0pt,minimum size=4.5mm}]
\node (n1) at (0,0)   [process] {$v_{s}$};
\node (n2) at (1.05,0) [process] {$v_{1}$};
\node (n3) at (2.62,0) [process] {$v_{2}$};
\node (n4) at (4.3,0) [process] {$v_{3}$};
\node (n5) at (5.95,0) [process] {$v_{4}$};
\node (n6) at (7,0)  [process] {$v_{t}$};
  		
\draw[dashed][ultra thick] (n1) to   (n2);
\draw[ultra thick] (n2) to  node[midway,above](){$e_{p}$} (n3);
\draw[ultra thick] (n3) to  node[midway,above](){$e_{i}$} (n4);
\draw[ultra thick] (n4) to  node[midway,above](){$e_{s}$} (n5);
\draw[dashed][ultra thick] (n5) to  (n6);
\end{tikzpicture}}
\caption{Path $\mathcal{P}_{s,t}$}
\label{fig:part-c}
\end{figure}

{\noindent Finally, we show the claim~\eqref{eq:claim} is true by considering the two possible orders of edge-adding:}

{\noindent \bf {\em i})} The edges are added in the order $e_s$\textrightarrow$e_p$\textrightarrow$e_i$, the joining time for $e_p$ and $e_i$ are: 
\begin{align*}
\begin{cases}
t_{p,k-1}^{\jn} \!&=\! \Big(|\Ts_k|l^{(s)}_{v_{2}} 
-\!\!\!\!\!\sum\limits_{v\in T^{(s)}_k}\!\!\! l^{(s)}_v\Big)_{,}^{-1}\\[2em]
t_{i,k-1}^{\jn}\!&=\!\Big(|\Tt_k|l^{(t)}_{v_{2}}-\!\!\!\!\!\sum\limits_{v\in T^{(t)}_k}\!\!\!l^{(t)}_v\Big)_{.}^{-1}    
\end{cases}
\end{align*}
The assumption that $e_p$ is added before $e_i$ implies $t_{p,k-1}^{\jn} >t_{i,k-1}^{\jn} $ concluding the claim~\eqref{eq:claim}.

{\noindent \bf {\em ii})} The edges are added in the order $e_p$\textrightarrow$e_s$\textrightarrow$e_i$, the joining time for $e_p$ and $e_s$ are:   
\begin{align*}
\begin{cases}
t_{p,k-2}^{\jn}\!&=\!\Big(|\Ts_k|l^{(s)}_{v_{2}} 
-\!\!\!\!\!\sum\limits_{v\in T^{(s)}_k}\!\!\!l^{(s)}_v\Big)^{-1}_{,}\\[2em]
t_{s,k-2}^{\jn}\!&=\!\Big(|\Tt_k|l^{(t)}_{v_{2}} 
-\!\!\!\!\!\sum\limits_{v\in T^{(t)}_k}\!\!\!l^{(t)}_v\Big)_{.}^{-1}        
\end{cases}
\end{align*}
The order $e_p$ is added before $e_s$ concludes the claim~\eqref{eq:claim} because $t_{p,k-1}^{\jn} >t_{s,k-1}^{\jn}$.

\noindent The proof that the crossing times for the edges that belong to the tree $\calA_k^{(t)}$ is by symmetry and interchanging $s$ and $t$. 
\end{proof}

We now establish a parallel
between the LARS algorithm and the bi-directional Dijkstra algorithm, in that they share the defining feature, terminating when two trees built from opposite directions meet.

\begin{theorem}[\bf Equivalence]\label{thm:thm1}
The LARS algorithm iteratively builds two shortest-path trees, starting from roots $v_s$ and $v_t$, and terminates when these two trees connect.
\end{theorem}

\begin{proof}
Assuming the induction hypothesis in Assumption~\ref{assum:hypothesis} holds true, we have shown that in the $k$th iteration: i) edges connecting to vertices with the shortest distance to $v_s$ and $v_t$ are added to the active set (Lemma~\ref{lemma:addactive}); ii) crossing, where edges are removed from the active set, does not occur when $\lambda > 0$ (Lemma~\ref{lemma:cross}). Finally, iii) the algorithm terminates when the two shortest-path trees connect and $\lambda = 0$ (Lemma~\ref{lemma:terminate}).
\end{proof}

\section{Proximal algorithm for large-scale graph}\label{sec:proximal}
The LARS algorithm may provide a computational advantage if the desired path only contains a few edges. For large graphs, the number of breakpoints (proportional to the number of edges in the active set), tends to be very large, rendering $\beta(\lambda)$ intractable. The ADMM algorithm, on the other hand, is particularly well-suited for solving large-scale convex optimization problems in a distributed manner due to its scalability~\cite{boyd2004convex,boyd2011distributed,yue2018implementing}.

\subsection{ADMM and InADMM algorithm}
Application of the ADMM to the lasso, as presented in~\cite[Section 6.4]{boyd2011distributed}, is based on the reformulation of the lasso  problem~\eqref{eq:lasso} as follows: 
\begin{equation}\label{eq:lasso-new}
\min_{\beta, \alpha \in \mathbb R^m}\frac{1}{2}\|y-Q\beta\|_2^2 +\lambda\|\alpha\|_1 +\frac{\rho}{2}\|\beta-\alpha\|^{2}_{2},~\text{s.t}~\alpha\!=\!\beta,
\end{equation}
where $\alpha \in \mathbb{R}^m$ is an additional optimization variable, and $\rho$ is a positive constant. 

The {\em Lagrangian} $L_{\rho}(\beta,\alpha,u)$ corresponding to the constrained optimization problem~\eqref{eq:lasso-new} is 
\begin{align*}
L_{\rho}= \frac{1}{2}\|y-Q\beta\|_{2}^{2}+\lambda\|\alpha\|_{1} + u^{T}(\beta-\alpha)+\frac{\rho}{2}\|\beta-\alpha\|^{2}_{2},
\end{align*}
where $u\in \mathbb R^m$ is the {\em Lagrange multiplier}. Let $v \triangleq u/\rho$, the ADMM algorithm consists of the $\alpha,\beta$-minimization steps and the step updating the dual variable $v$. Specifically, the optimal variables $\beta,\alpha,v$ is computed as
\begin{align}
\beta^{k} :&={\rm arg}\min_{\beta} L_{\rho}(\beta,\alpha^{k-1},v^{k-1})\nonumber\\
&= (Q^{T}Q+\rho I)^{-1}\Big(Q^{T}y+\rho(\alpha^{k-1}-v^{k-1})\Big) \label{eq:beta-update}\\    
\alpha^{k} :&= {\rm arg}\min_{\alpha} L_{\rho}(\beta^{k},\alpha,v^{k-1})= S_{\lambda/\rho}\Big(\beta^{k}+\frac{1}{\rho}v^{k-1}\Big)\nonumber\\
v^{k} :&= v^{k-1}+\rho(\beta^{k} - \alpha^{k}), \nonumber
\end{align}
in the $(k-1)$th iteration.The $\beta$-update can be found in~\cite[Section 4.2]{boyd2011distributed} and $S_{\lambda/\rho}$ is the element-wise interpreted proximity operator of the $\ell_1$ norm, known as the {\em soft-thresholding} operator in~\cite[Section 4.4.3]{boyd2011distributed}.

To reduce the size and complexity of ADMM, we replace the matrix inversion $(Q^{T}Q+\rho I)^{-1}$
by
\begin{equation}\label{eq:matrix-identity}
(Q^TQ + \rho I)^{-1} = \frac{1}{\rho}\Big(I - Q^T\big(QQ^T + \rho I\big)^{-1}Q\Big),
\end{equation} 
using the {\em matrix identity}, which instead involves $(Q^{T}Q+\rho I)^{-1}$. We will show this step significantly reduces the complexity of the algorithm in the next section. To further reduce the complexity of ADMM for large-scale graphs, we use the InADMM algorithm introduced in~\cite{yue2018implementing} whose key idea is to approximately solve a system of linear equations instead of evaluating the matrix inversion exactly, using the matrix identity~\eqref{eq:matrix-identity} to replace the $\beta$ update in~\eqref{eq:beta-update} with
\begingroup
\allowdisplaybreaks
\begin{align*}
& h^{k-1} := Q^{T}y+\rho(\alpha^{k-1} - v^{k-1})\\
&\eta^{k} := (QQ^{T}+\rho I)^{-1}Q h^{k-1}, \\
&{\beta}^{k} := \frac{1}{\rho} (h^{k-1} - Q^{T}\eta^{k})
\end{align*}
\endgroup
and computes $\eta^{k}$ approximately using the {\em conjugate gradient} (CG) method~\cite{hestenes1952methods}.

\blue 
\begin{remark}[\bf Stopping criteria]
Both ADMM and InADMM use the standard primal--dual residual test. With $r_k=\beta_k-\alpha_k$ and $s_k=\alpha_k-\alpha_{k-1}$, a stopping criterion is when $\|r_k\|_2 \le \varepsilon_{\mathrm{abs}}\sqrt{m} + \varepsilon_{\mathrm{rel}}\max\{\|\beta_k\|_2,\|\alpha_k\|_2\}$, or when $\|s_k\|_2 \le \varepsilon_{\mathrm{abs}}\sqrt{m} + \varepsilon_{\mathrm{rel}}\|v_k\|_2$. We use $(\varepsilon_{\mathrm{abs}},\varepsilon_{\mathrm{rel}})=(10^{-8},10^{-6})$, with $m$ the number of edges. In InADMM, the inner CG solver uses relative residual $10^{-8}$ (up to $2\times 10^3$ iterations).
\end{remark}
\black

\subsection{Warm-start, Parallel lasso, and complexity}
The lasso formulation and the ADMM solvers admit remark features that can not be found in A$^\star$ search algorithm. We postulate two key features that are brought into the shortest path problem.

Firstly, distinguished from Dijkstra's algorithm, the ADMM allows the input of any initializer (potential-path), that may be used in tracking the graph topological changes. The problem then becomes correcting/updating the shortest path with an initial guess/prior as inputs. Our approach shows a fast convergence under such circumstances, see Fig. \ref{fig:Convergence}. 

Notably, the problem can be generalized through {\em parallel lasso}, addressing the multi-pair/all-pair shortest path problem defined as:
\begin{equation}
{\rm arg}\min_{\beta \in \mathbb R^m} \Big\{\sum_{i = 1,\dots,N}\frac{1}{2}\| \mathbf y_i-Q \beta_{(i))}\|_2^2 + \lambda\|\beta_{(i)}\|_1\Big\}\label{eq:parallel}
\end{equation}
where the $i$-th column of $\mathbf y$ is the indicator vector of path, and thus we defined $\mathbf y_{1,\dots,N}=[y_{(1)},\dots,y_{(N)}]$ and $\beta = [\beta_{(1)},\dots,\beta_{(N)}]$. The framework \eqref{eq:beta-update} is thus solves a multi-pair shortest path problem. The optimal variables in ADMM can be easily extended vector-wise while the dominated matrix inverse only needs to be computed once for all. Specifically, when $\mathbf y$ characterizes all pairs of vertices, Problem \eqref{eq:parallel} becomes the lasso formulation of all-pairs shortest path problem.

We use {\em worst-case complexity} (denoted as $\calO(\cdot)$) to evaluate the algorithms' performance. It quantifies the operations of the algorithm in the worst case, see \cite[Appendix C]{boyd2004convex}, \cite{cormen2022introduction}. The complexity of ADMM is dominated by the matrix inversion $(Q^{T}Q+\rho I)_{m\times m}^{-1}$, which is of order $\calO(m^3)$ (e.g. using {\em Cholesky decomposition}). However, by utilizing the matrix identity~\eqref{eq:matrix-identity}, this complexity is reduced to $\mathcal O(n^3)$. The most costly step in the CG method is the matrix-vector multiplication $(Q Q^{T} + \rho I)x$ where $x \in \mathbb R^n$. This has a complexity of $\mathcal O(m)$ because the weighted incidence matrix $Q$ has $2m$ nonzero elements. Assuming the CG algorithm terminates in $T_{CG}$ iterations, the complexity of the CG step in InADMM algorithm is of order $\calO(mT_{CG})$. Noticing the complexity of other operations in InADMM is at most $\calO(m)$, we summarize the complexities in Table~\ref{tab:complexity} as below.
\begin{table}[htb!]
\centering 
\begin{adjustbox}{width=.6\columnwidth}
\begin{tabular}{ccccc} 
\hline \hline 
{Variables}  & $\eta$ & $\beta$(Cholesky) & $\alpha$ & $v$ \\ [0.5ex]
\hline  
ADMM\footnote[6]     & $\calO(nm)$ &$\calO(n^{3})$ & $\calO(m)$ & $\calO(m)$\\
InADMM\footnote[7]   & $\calO(mT_{CG})$ & $\calO(m)$ & $\calO(m)$ & $\calO(m)$\\[1ex] 
\hline
\end{tabular}
\end{adjustbox}
\caption{  Complexities of ADMM and InADMM per iteration.}
\label{tab:complexity}
\end{table}

\footnotetext[6]{Per iteration, both methods converge within 40 steps for a graph with $17291$ edges, see Fig. \ref{fig:Convergence}.}
\footnotetext[7]{$T_{CG} \ll n$ and scalable with respect to the size of the graph. Empirically, we have $T_{CG}= 350$ for graph with $4422$ vertices and $T_{CG}= 320$ for graph with $5694$ vertices. To reach $\epsilon$ accuracy, using $O(1/\epsilon)$ iteration complexity \cite[Theorem 4.1]{yue2018implementing}, we conclude overall $O(nm /\epsilon)$ and $O(n^3/\epsilon)$ complexity for update of $\eta$ in InADMM and $\beta$ in ADMM, respectively.}

\subsection{\blue Example 1: Intelligent scissors on image}\label{subsec:num-exp}

We \blue first consider \black ``\textit{intelligent scissors}'', an image segmentation technique that is commonly used in computer vision~\cite{mortensen1995intelligent}. It creates a {\em grid graph} over the image with each pixel represented by a vertex, and is connected to its $8$ neighboring pixels through edges. The edge weights are assigned by a cost function~\cite[Section 3]{mortensen1995intelligent} based on image features. The objective is to identify a boundary between the portraits and the background, which turns out to be equivalent to finding the shortest path between selected pixels (marked in green) over the generated graph. The segmentation is obtained by identifying the edges in the shortest path. Herein, we use an image with $4422$ pixels. The edges, detected by constructing shortest paths using three different methods and colored in red, are drawn in Fig.~\ref{fig:IntelSci} (\blue The code used in the paper is available at~{ \url{github.com/dytroshut/Lasso-shortest-path.git};  functions in the code can be found in \cite{boydlassocode2011})}.\black 
For the InADMM, we utilized the CG method from~\cite{hestenes1952methods,yue2018implementing}.
The convergence rates of ADMM, InADMM, and Basis Pursuit\footnote[8]{  Basis Pursuit is a technique to obtain sparse solution to an underdetermined system of linear equations.}~\cite[Section 6.1]{boyd2011distributed}, are shown in Fig.~\ref{fig:Convergence}.  Further, to highlight the efficiency of lasso's distributed implementation, we used InADMM (with a pre-specified path as input).
\begin{figure}[htb!]
\centering
\begin{subfigure}{0.3\columnwidth}	
\centering
\includegraphics[width=\linewidth]{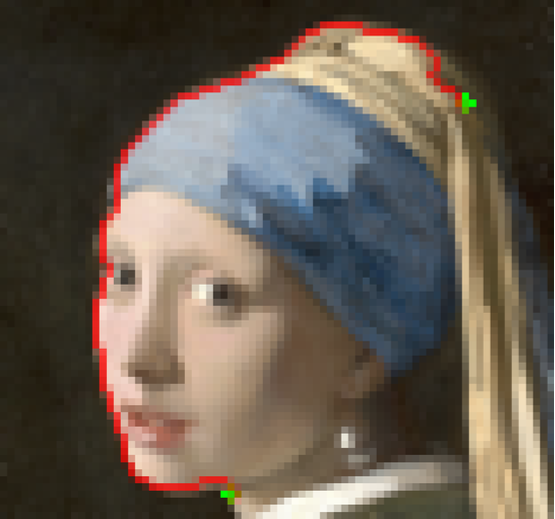}\vspace{-.7pt}
\includegraphics[width=\linewidth]{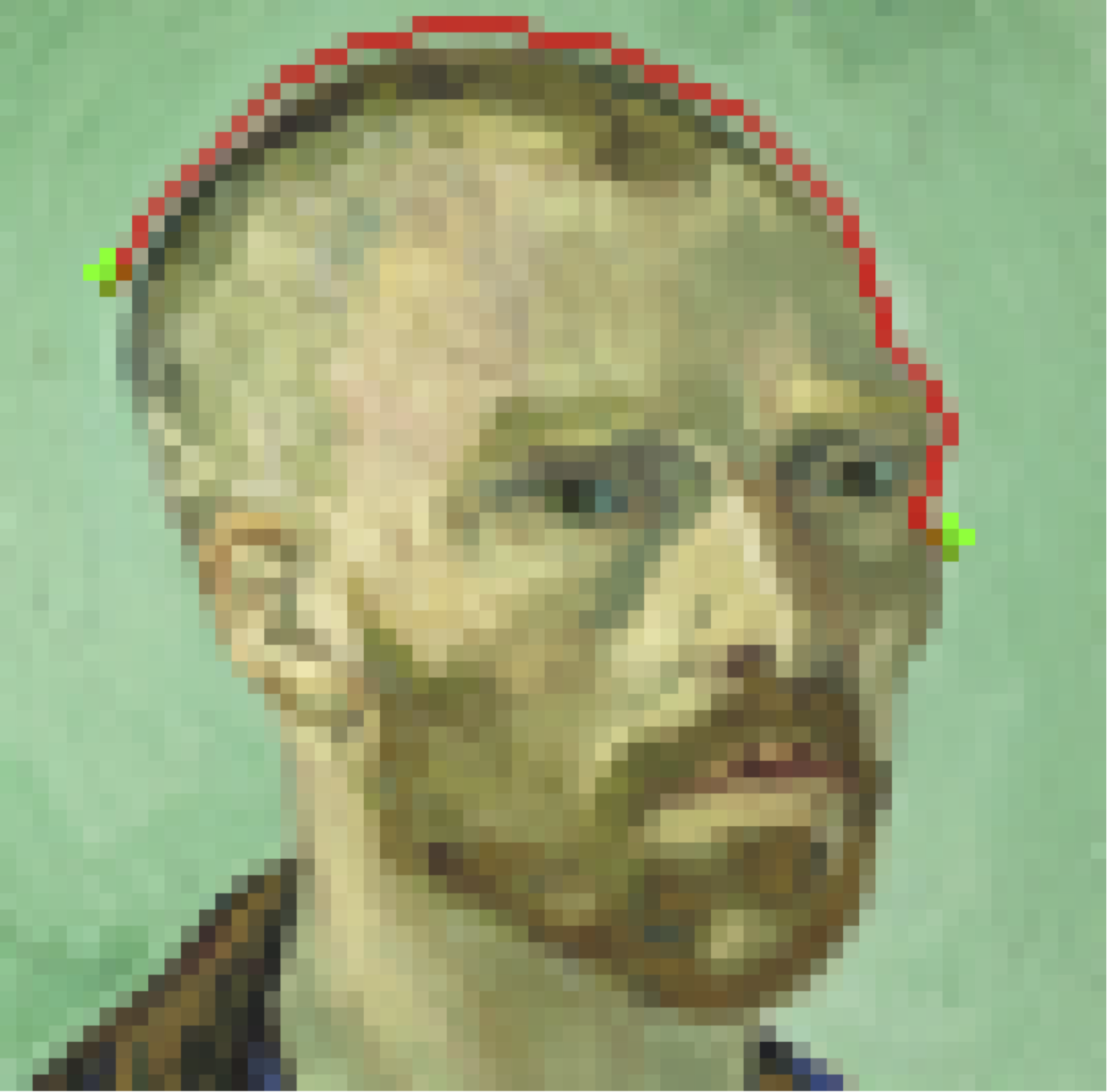}
\caption{Dijkstra}
\label{fig:Dijkstrapath}
\end{subfigure}\hspace{-2.3pt}
\begin{subfigure}{0.3\columnwidth}
\centering
\includegraphics[width=\linewidth]{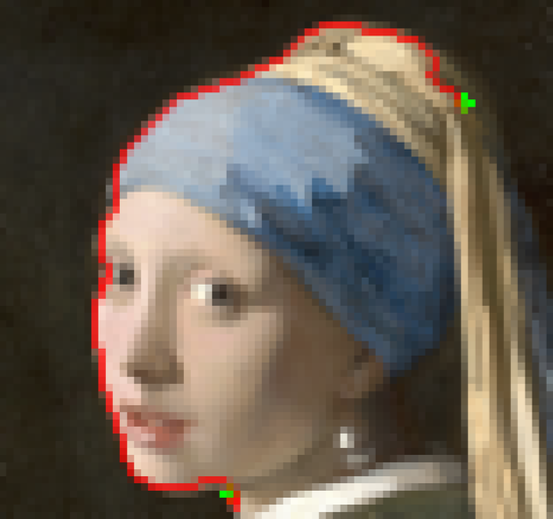}\vspace{-.7pt}
\includegraphics[width=\linewidth]{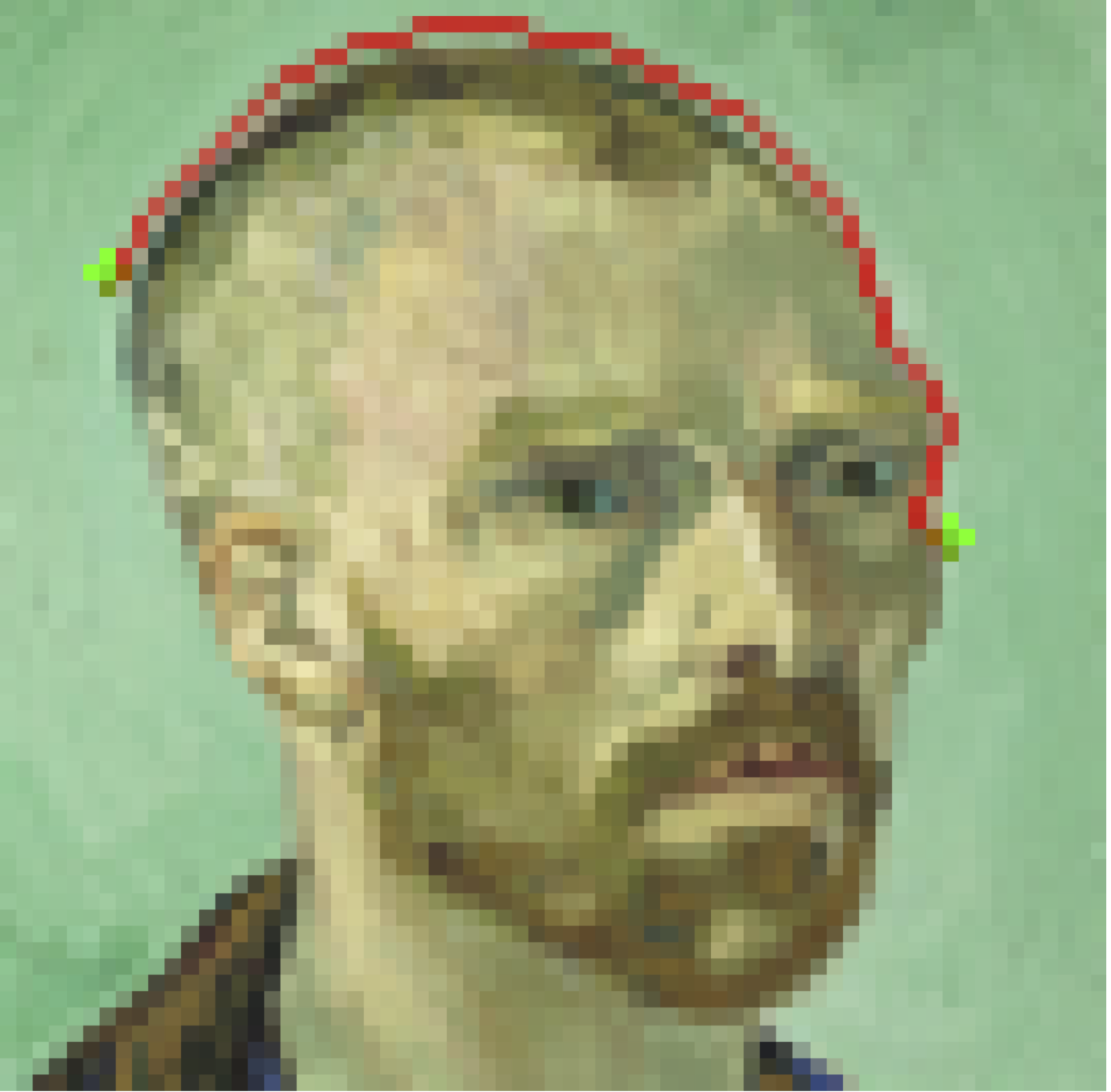}
\caption{ADMM}
\label{fig:ADMMpath}
\end{subfigure}\hspace{-2.3pt}
\begin{subfigure}{0.3\columnwidth}
\centering
\includegraphics[width=\linewidth]{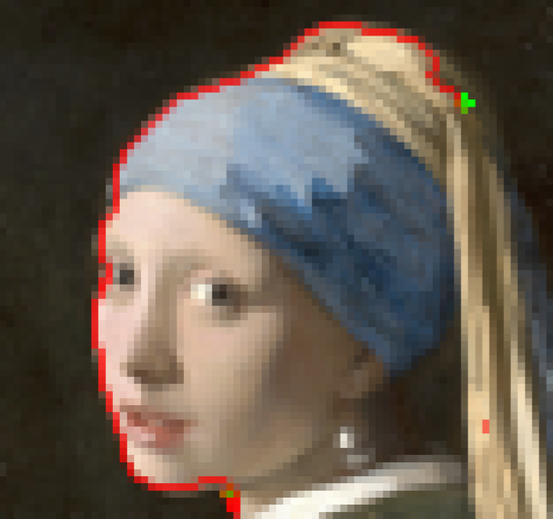}\vspace{-.7pt}
\includegraphics[width=\linewidth]{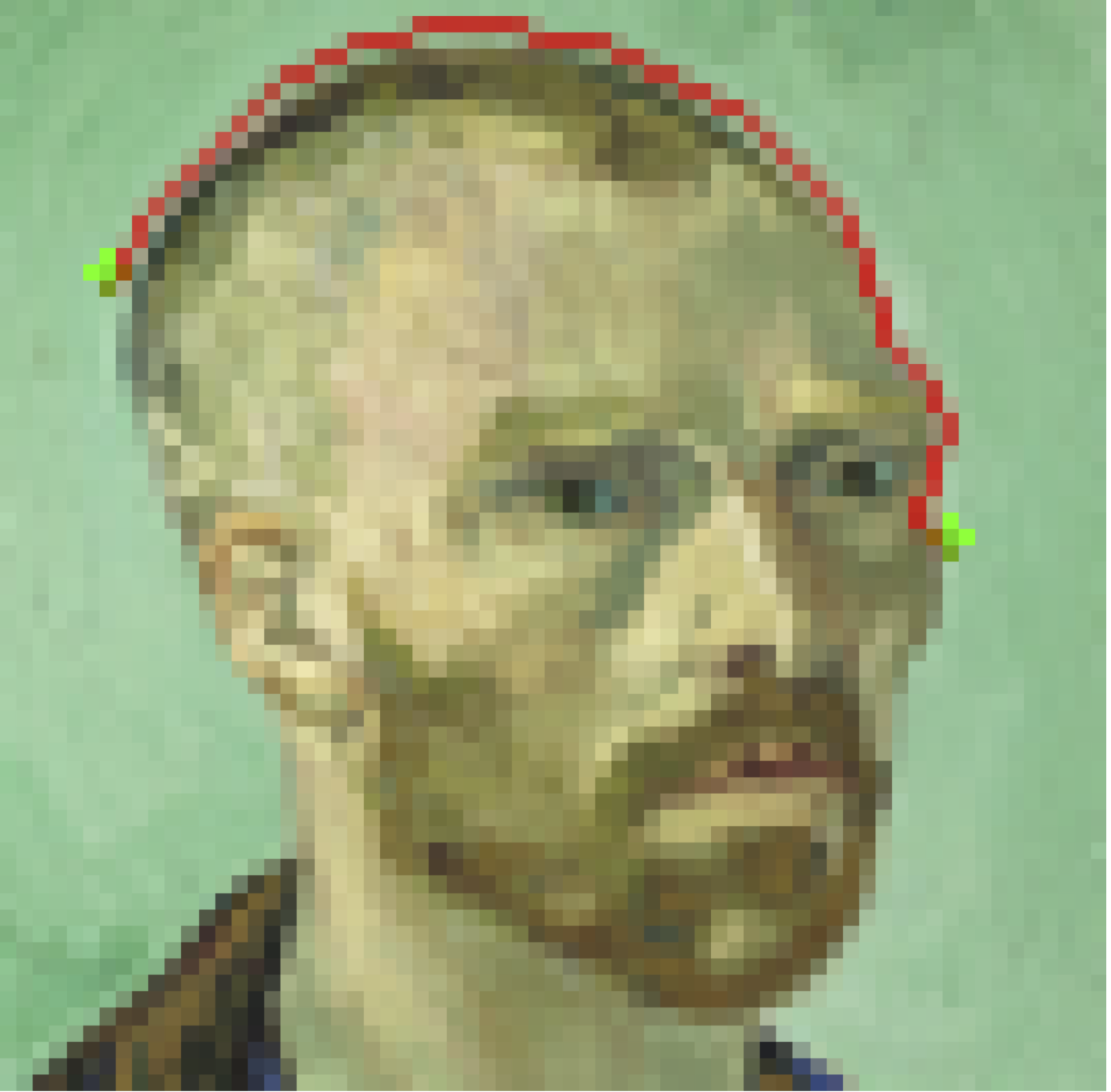}
\caption{InADMM}
\label{fig:InADMMpath}
\end{subfigure}
\caption{``Edge detection'' in an image as a short path, highlighted in red, and obtained using Dijkstra’s algorithm (Fig.~\ref{fig:Dijkstrapath}), as well as using the lasso solution $\beta$ in ADMM and InADMM in~Fig.~\ref{fig:ADMMpath} and Fig.~\ref{fig:InADMMpath}, respectively.}
\label{fig:IntelSci}
\end{figure}

\begin{figure}[htb!]
\centering
\includegraphics[width=0.8\columnwidth]{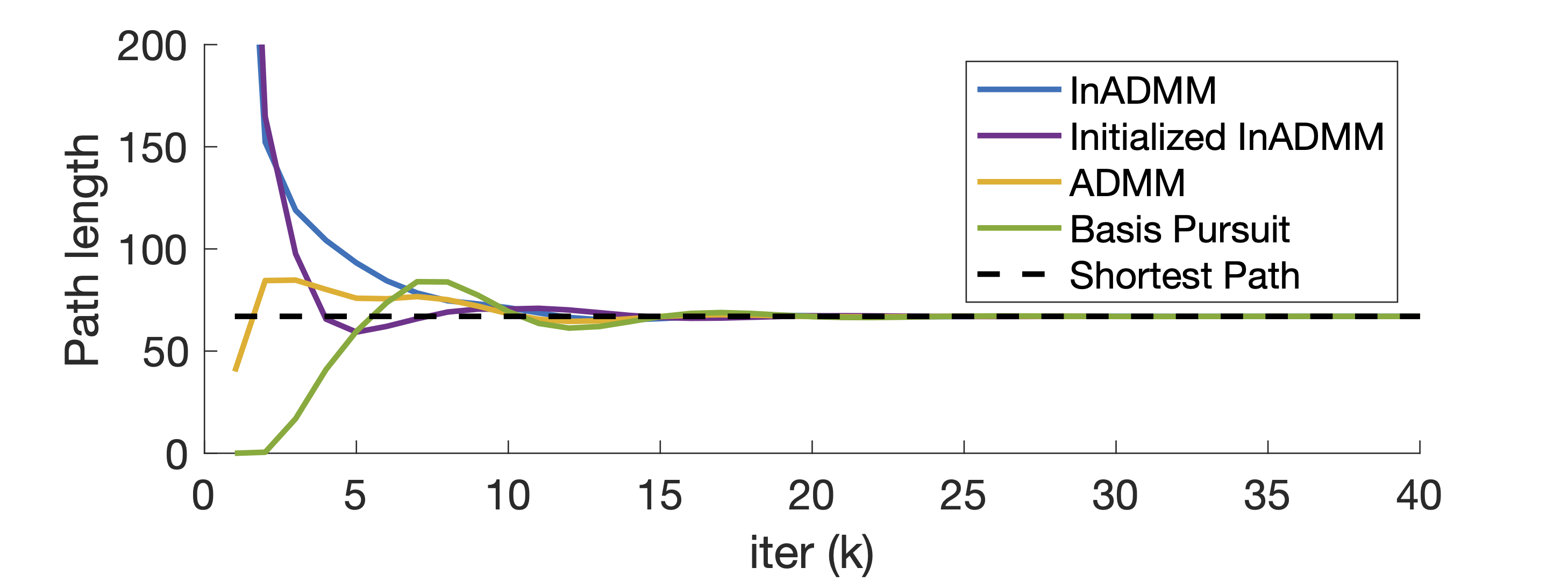}
\caption{Example based on the image ($ n = 4422$ and $m = 17291$): ``Edge in image'' identified via InADMM, InADMM with initializer, ADMM, and Basis pursuit. These converge in $36$, $34$, $29$, $47$ steps, respectively, with running times $1.7308$, $1.7618$, $4.0249$, $35.7527$ seconds in MATLAB clock.
Also, the running times for solving the linear program \eqref{eq:linprog1} using three methods ({\em dual-simplex and interior-point(-legacy)}) are $26.6145$, $29.3866$, $32.2845$ seconds.}
\label{fig:Convergence}
\end{figure}

\subsection{\blue Example 2: Route planning on road networks}\label{subsec:roads}
\blue
Next, we study \emph{``route planning''} on urban road maps, a standard task in transportation. This induces an \emph{undirected road graph} with intersections as vertices and street segments as edges. Edge weights are travel times, computed from segment length and nominal speed, with turn costs at busy junctions. The minimum-time trip between two addresses is then the shortest path on this weighted graph. To highlight contrasting geometries, we use city-scale OpenStreetMap \cite{yap2023global} extracts for Athens ($n{=}1009$, $m{=}2340$) and Amsterdam ($n{=}1031$, $m{=}2037$). The recovered routes, computed by Dijkstra (yellow), ADMM (red), and InADMM (green), are shown in Fig.~\ref{fig:road}. These graphs are sparse and close to planar, so effective routes rise to arterial and return to local streets near the destination.

We build the signed incidence from the road graph and use travel–time edge weights. The regularization is set to $\lambda=10^{-4}\lambda_{\max}$, where $\lambda_{\max}$ is the standard lasso threshold. ADMM uses a penalty $\rho=1$ and over-relaxation $\alpha=1.8$ with a sparse Cholesky solver. InADMM keeps the same splitting and replaces the direct solver with a preconditioned conjugate-gradient step on the normal equations (diagonal preconditioner, relative tolerance $10^{-8}$, maximum $2000$ iterations). To avoid degeneracies and ensure a unique $s$–$t$ route, we apply a one-time strictly positive perturbation to all edge weights and use the same perturbed weights for Dijkstra, ADMM, and InADMM.

\black
\begin{figure}[htb!]
\centering
\setlength{\tabcolsep}{1pt}
\renewcommand{\arraystretch}{1.0}
\begin{tabular}{R{0.1\columnwidth} M{0.4\columnwidth} M{0.5\columnwidth}}
    \text{\small Dijkstra} &
    \includegraphics[height=0.29\columnwidth]{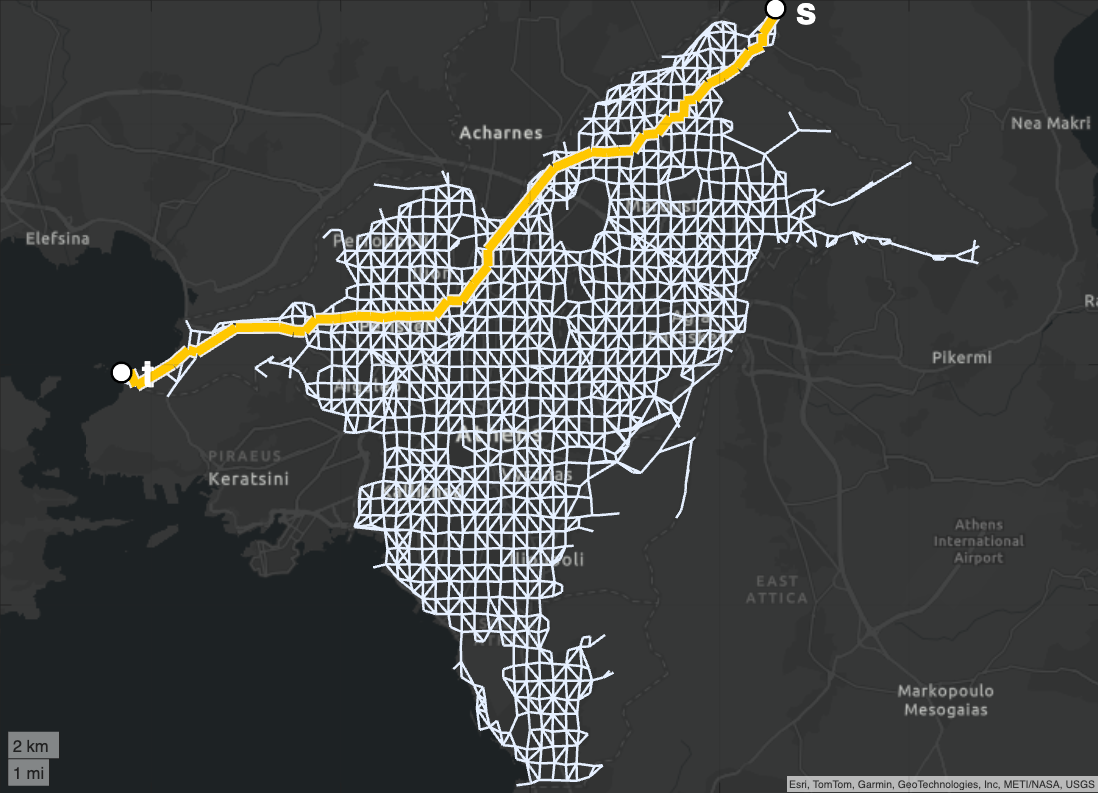} &
    \includegraphics[height=0.29\columnwidth]{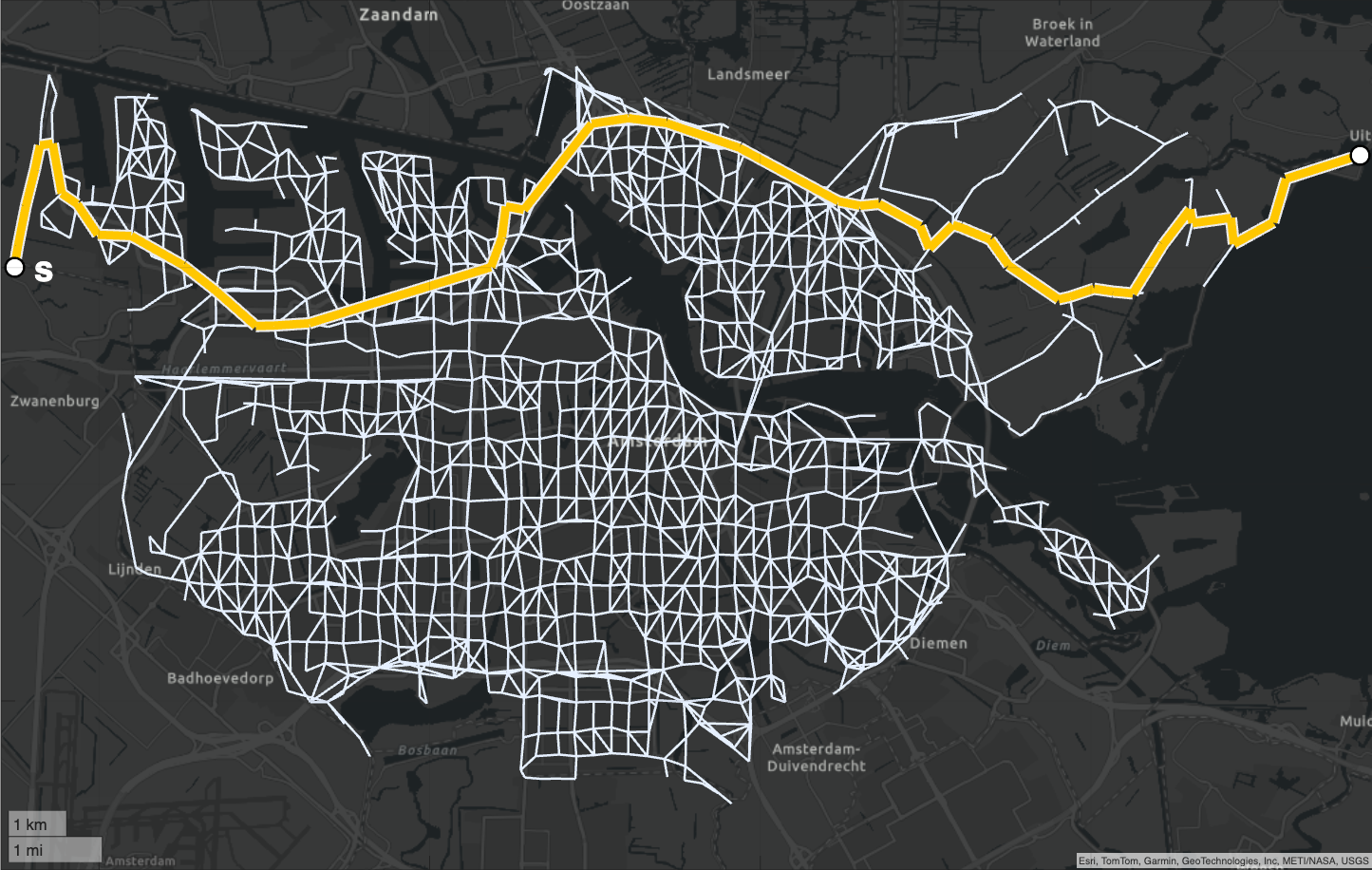} \\
    \text{\small ADMM} &
    \includegraphics[height=0.29\columnwidth]{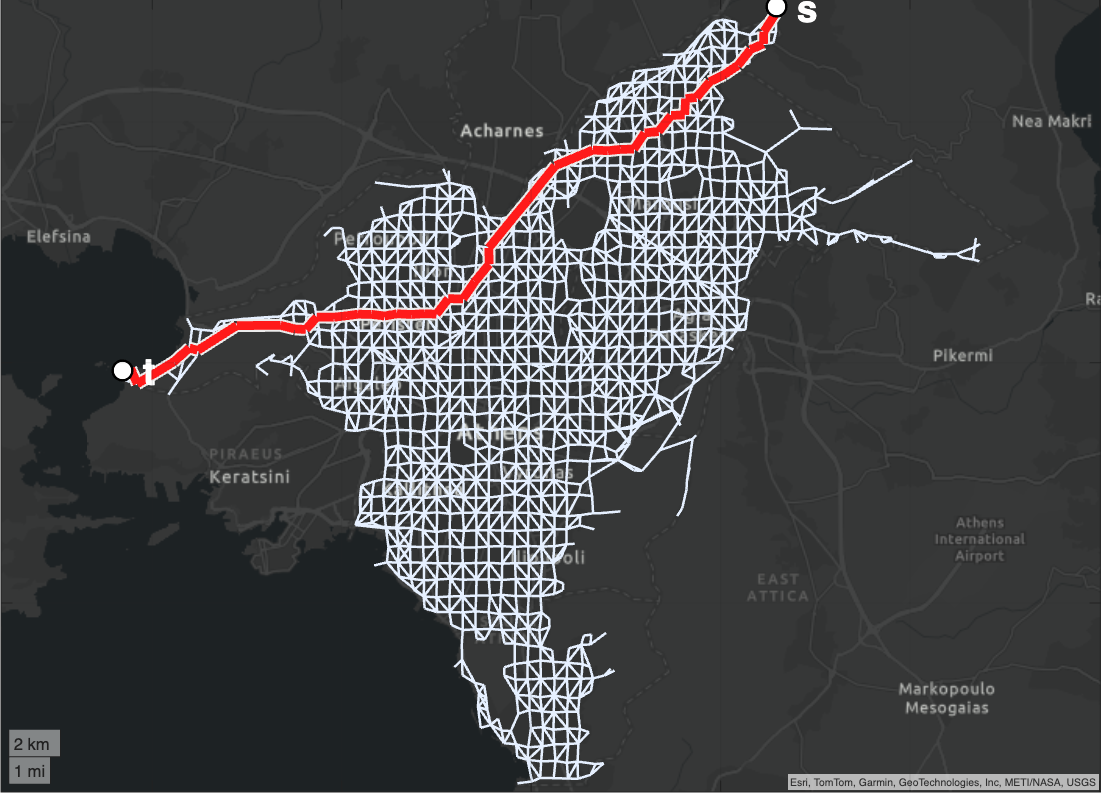} &
    \includegraphics[height=0.29\columnwidth]{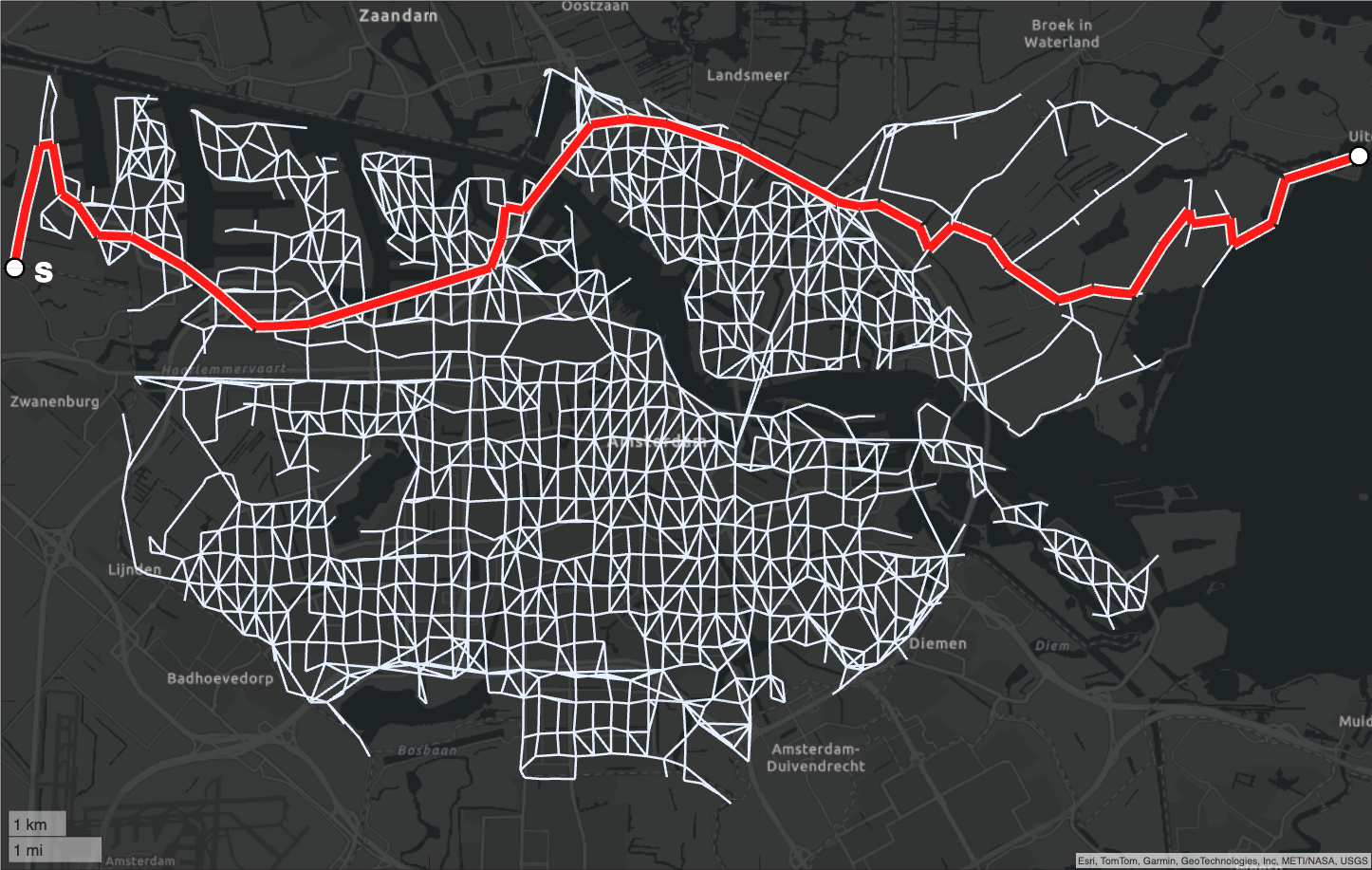} \\
    \text{\small InADMM} &
    \includegraphics[height=0.29\columnwidth]{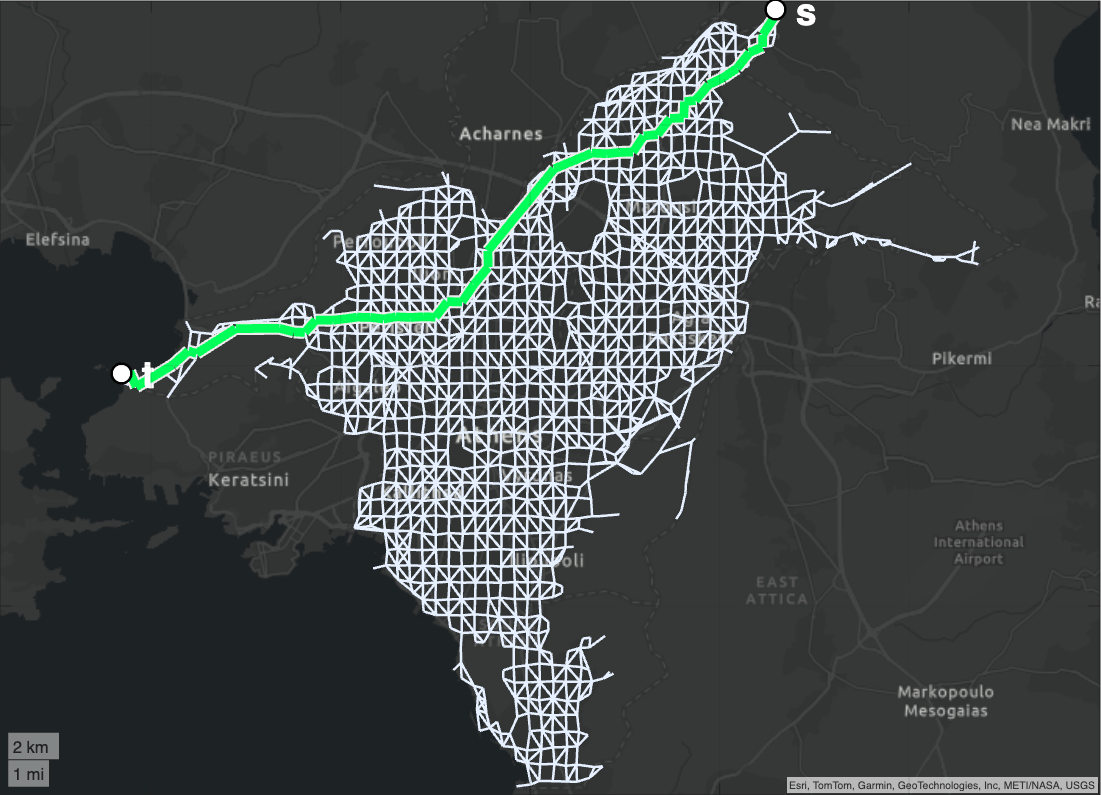} &
    \includegraphics[height=0.29\columnwidth]{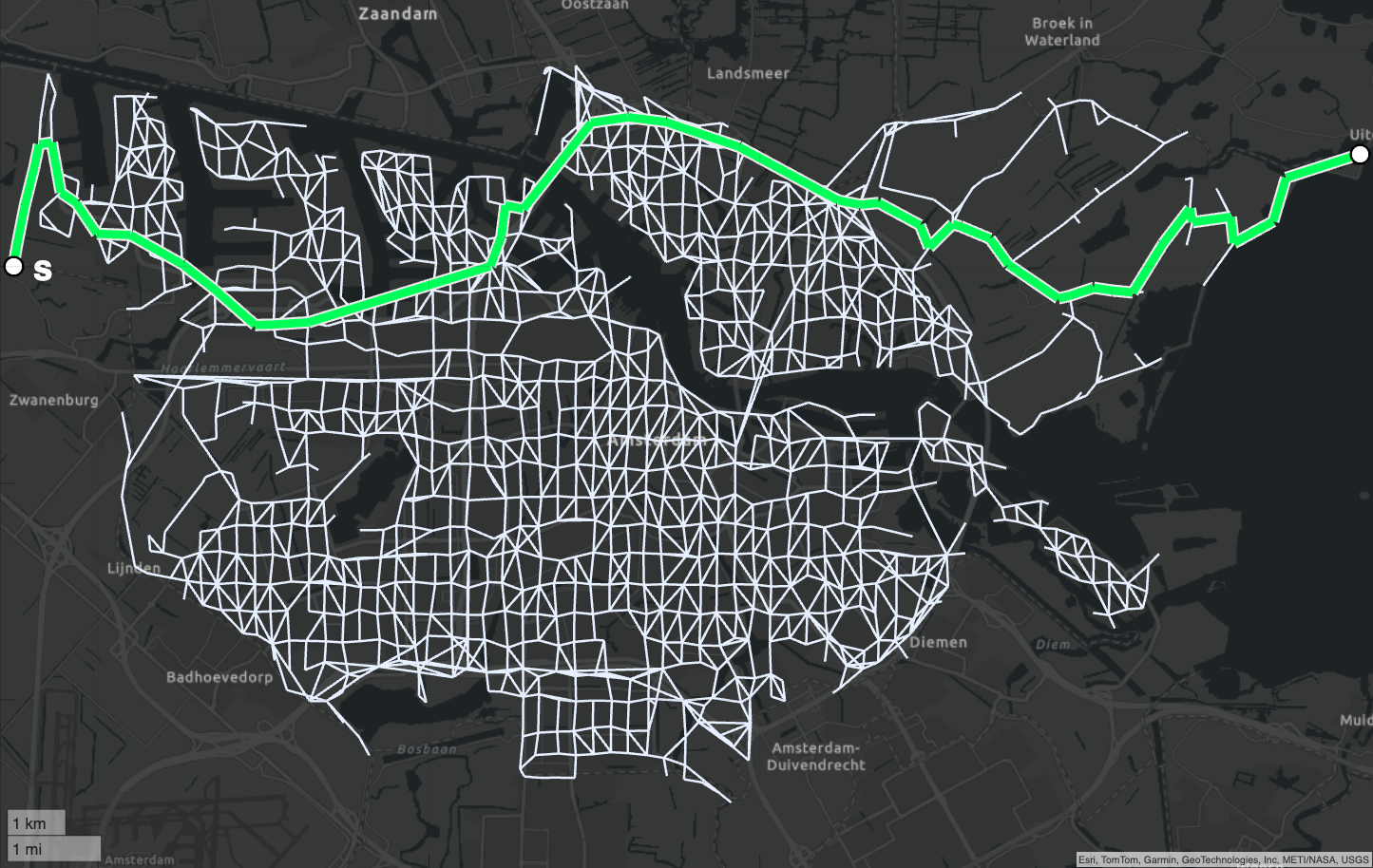} \\
\end{tabular}
\caption{\blue ``Minimum-time trip'' planning on Athens and Amsterdam road networks. Shortest paths from Dijkstra (yellow) and supports of the lasso solution $\beta$ recovered by ADMM (red) and InADMM (green) are overlaid.}
\label{fig:road}
\end{figure}
\blue 
Then, for scalability, we repeat the experiment on the Amsterdam map with finer segmentation at two larger sizes, $n=4699,\,m=8270$ and $n=9379,\,m=14571$ (see Fig.~\ref{fig:scalibility-road}). With the same settings as in Fig.~\ref{fig:road} and $\lambda=10^{-6}\lambda_{\max}$, the InADMM lasso path remains efficient and robust, recovering essentially the same minimum–time corridor on the higher–resolution graphs. A complementary effect appears when we decrease the penalty to $\lambda=10^{-7}\lambda_{\max}$. The solution begins to place nonzero weights on several nearly equivalent routes, so the highlighted support widens to reflect multiple plausible arterials. The phenomenon is more pronounced at higher resolution, where finer segmentation introduces additional near-equivalences between the same source and destination, as in Fig.~\ref{fig:scalibility-road}.
\black

\begin{figure}[htb!]
\centering
\setlength{\tabcolsep}{1pt}
\renewcommand{\arraystretch}{1.0}
\begin{tabular}{M{0.5\columnwidth} M{0.5\columnwidth}}
    \includegraphics[height=0.318\columnwidth]{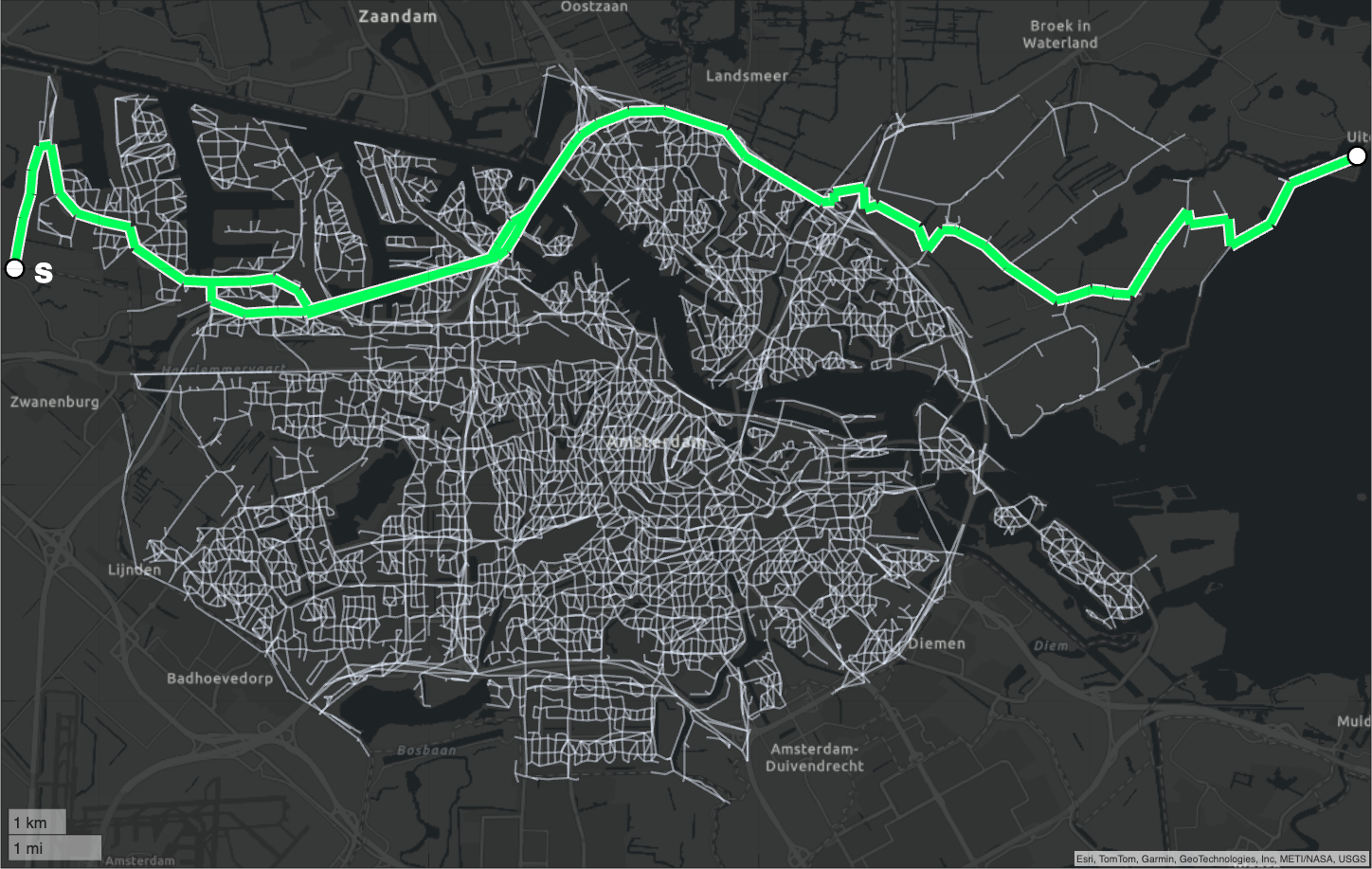} \subcaption{$\lambda=10^{-6}\lambda_{\max},\ n=4699$}&
    \includegraphics[height=0.318\columnwidth]{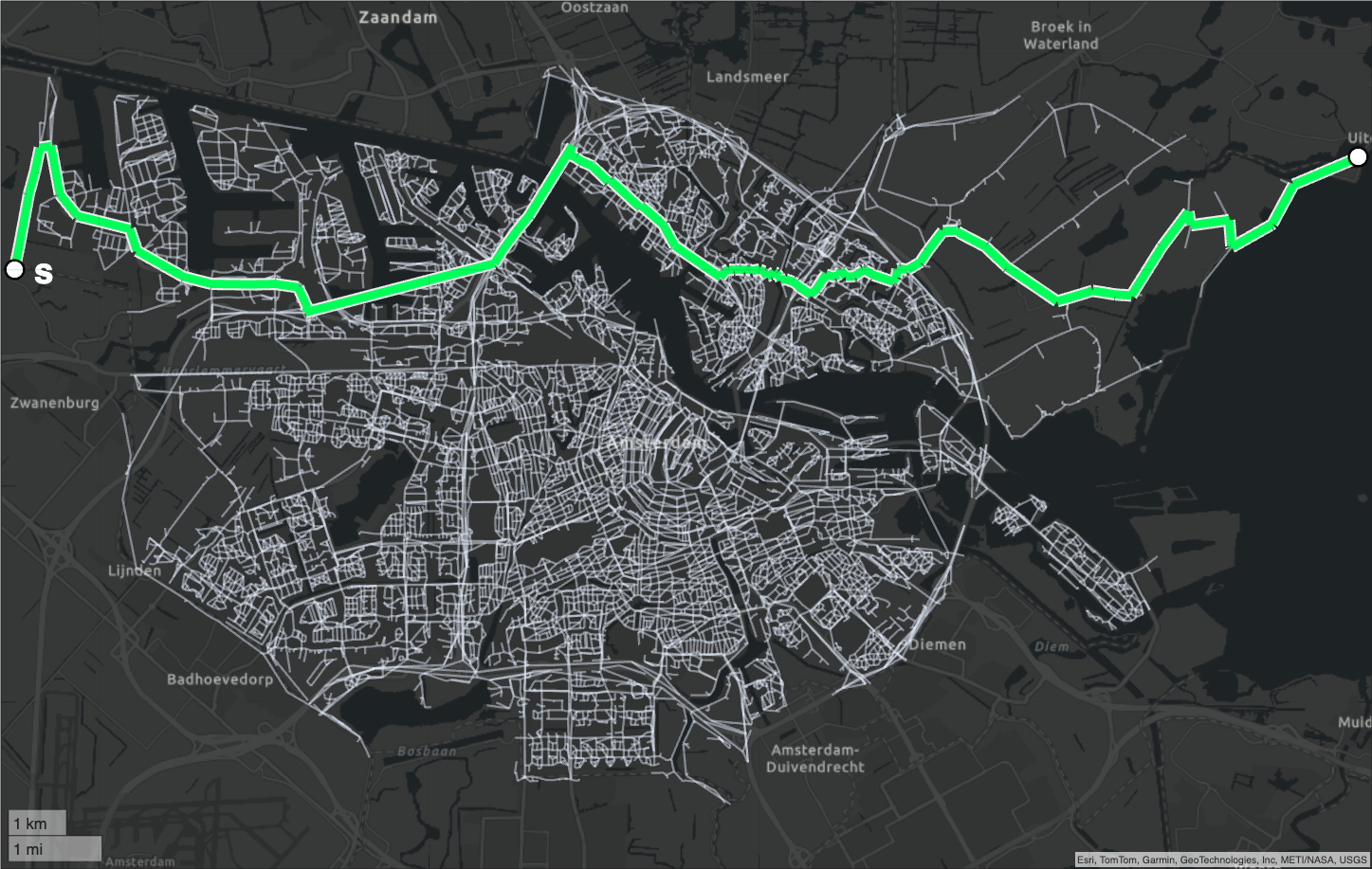} \subcaption{$\lambda=10^{-6}\lambda_{\max},\ n=9379$}\\[-1em]
    \includegraphics[height=0.318\columnwidth]{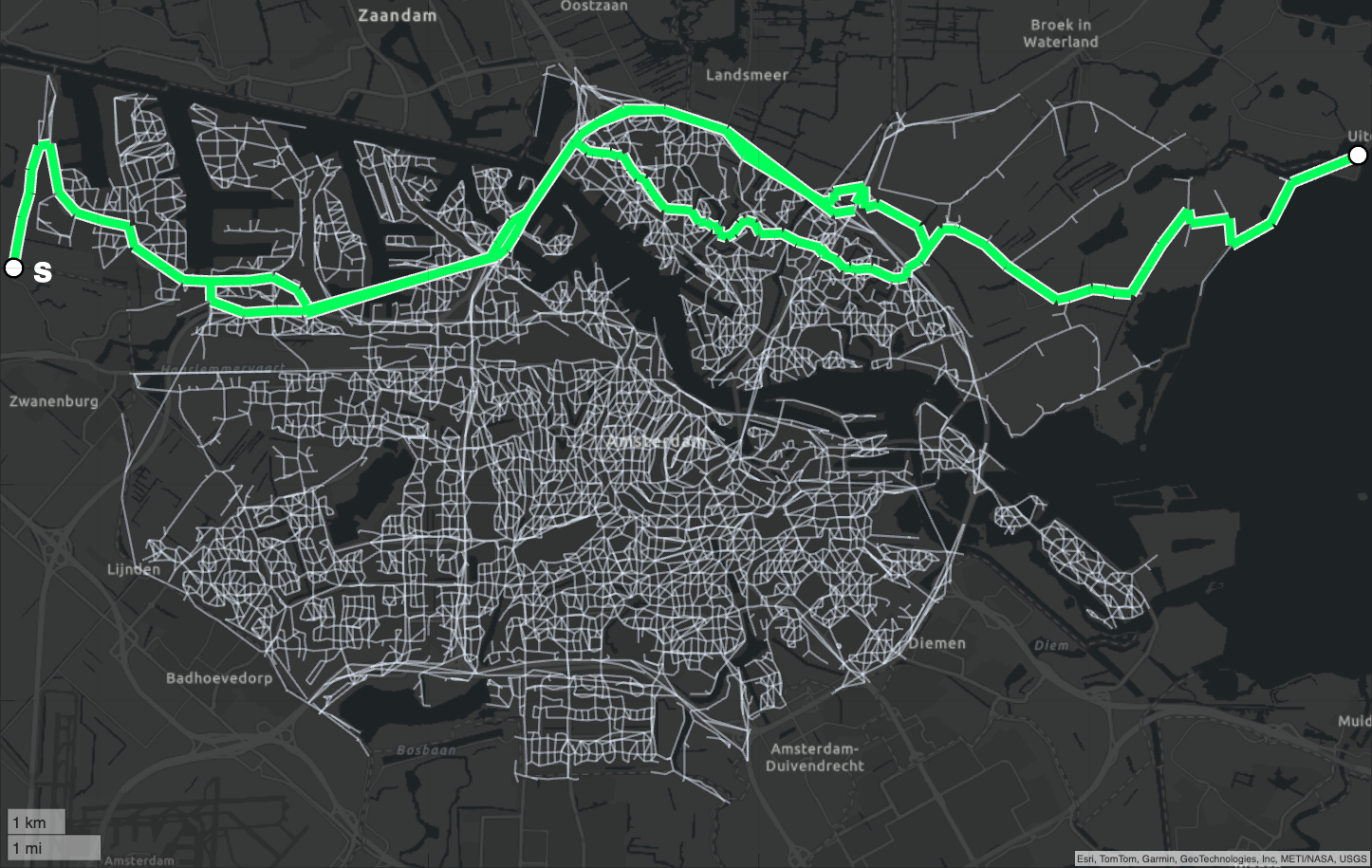} \subcaption{$\lambda=10^{-7}\lambda_{\max},\ n=4699$}&
    \includegraphics[height=0.318\columnwidth]{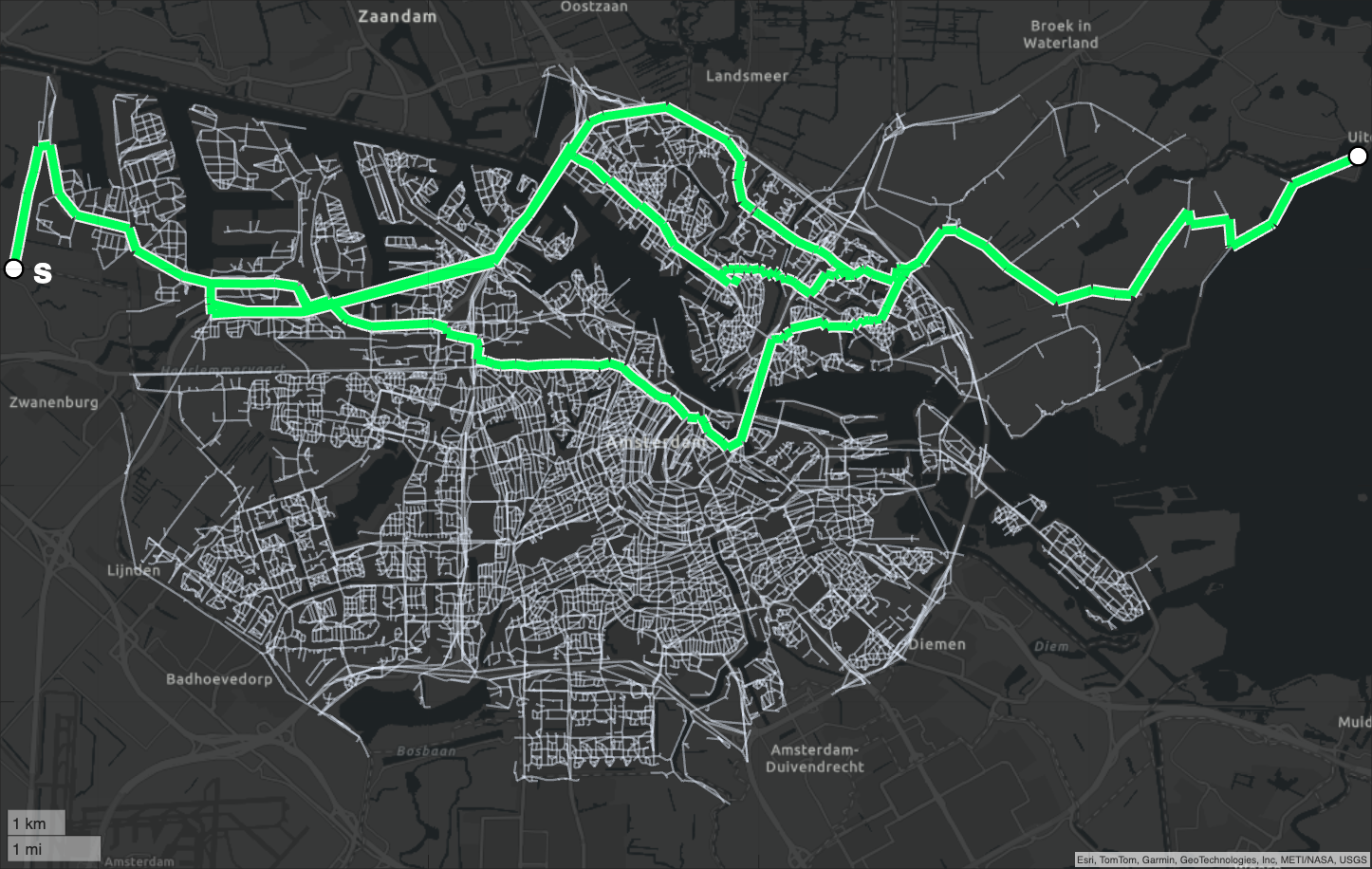} \subcaption{$\lambda=10^{-7}\lambda_{\max},\ n=9379$}\\
\end{tabular}
\caption{\blue Scalability and robustness on the same Amsterdam road map under finer segmentation. For $n=4699,\,m=8270$ and $n=9379,\,m=14571$, the InADMM lasso path with $\lambda=10^{-6}\lambda_{\max}$ recovers essentially the same minimum–time corridor. With $\lambda=10^{-7}\lambda_{\max}$, the support broadens, reflecting several nearly equivalent alternatives. All runs use identical weight jitter/tie–break and the same source–destination.}
\label{fig:scalibility-road}
\end{figure}

\subsection{\blue Example 3: Multi-hop shortest-path routing on random geometric networks}\label{subsec:rgg}
\blue 
Lastly, we build random geometric graphs for \emph{``multi-hop routing''}. We take the node set from the largest weakly connected component of the CiteSeerX network (IDs only, citation edges ignored) \cite{nr-aaai15}, place each node uniformly at random in $[0,1]^2$, and connect pairs whose Euclidean distance is at most a radius $r$. We increase $r$ until the giant component contains at least $98\%$ of the nodes, then keep that component ($n=2997$, $m=14149$).

Edge costs model transmit effort as squared Euclidean distance ($\alpha=2$). To avoid degeneracies and ensure a unique $s$–$t$ route, we apply a one-time strictly positive perturbation to all edge weights and use the same perturbed weights for Dijkstra, ADMM, and InADMM. The source $s$ is the left-most node; the target $t$ is the node with the largest finite weighted distance from $s$ under these weights, which yields a long left-to-right route that threads dense patches and skirts voids. All remaining solver settings (choice of $\lambda$, ADMM/InADMM parameters, and stopping criteria) match Example~2.
We then repeat the same construction at two larger scales, yielding
$(n,m)=(5997,\,30918)$ and $(n,m)=(8989,\,48819)$. As before, $s$ is chosen as the left-most node and $t$ as the node with the largest finite weighted distance from $s$. Across both sizes, the recovered lasso support (InADMM) coincides with the Dijkstra shortest path at our display thresholds, no side branches appear, and the route remains unique by construction.

\begin{figure}[htb!]
\centering
\begin{subfigure}{0.3\columnwidth}	
\centering
\includegraphics[width=\linewidth]{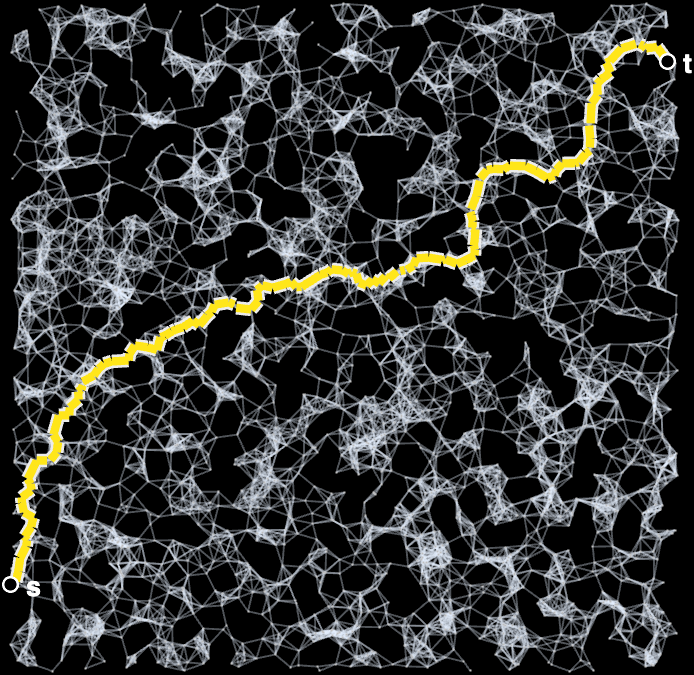}\vspace{-.7pt}
\includegraphics[width=\linewidth]{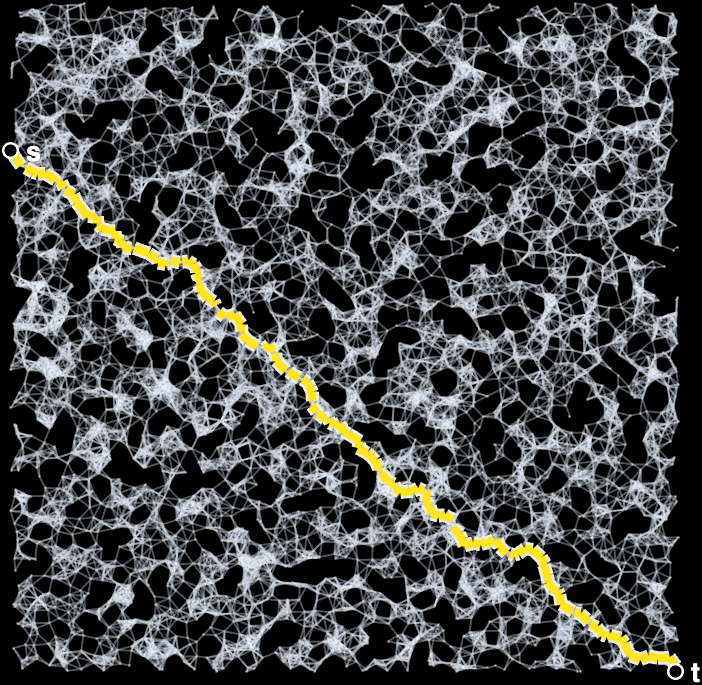}\vspace{-.7pt}
\includegraphics[width=\linewidth]{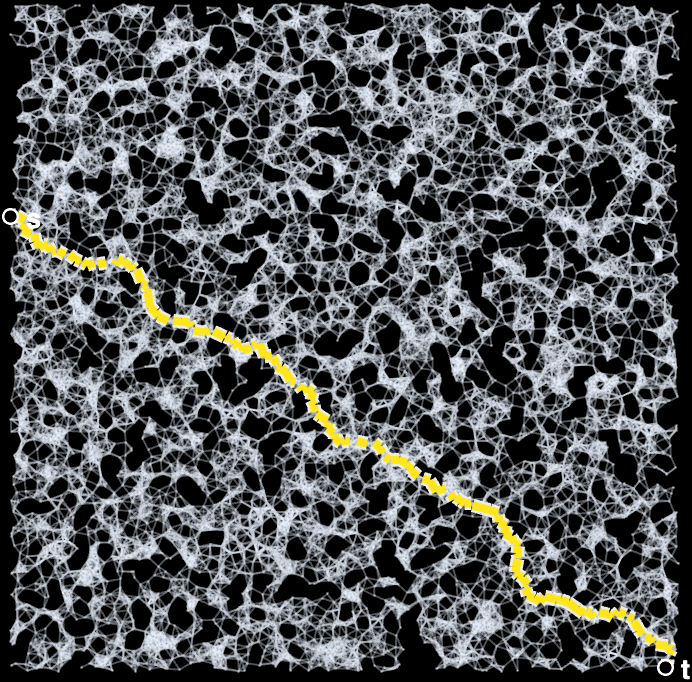}
\caption{Dijkstra}
\end{subfigure}\hspace{-2.3pt}
\begin{subfigure}{0.3\columnwidth}
\centering
\includegraphics[width=\linewidth]{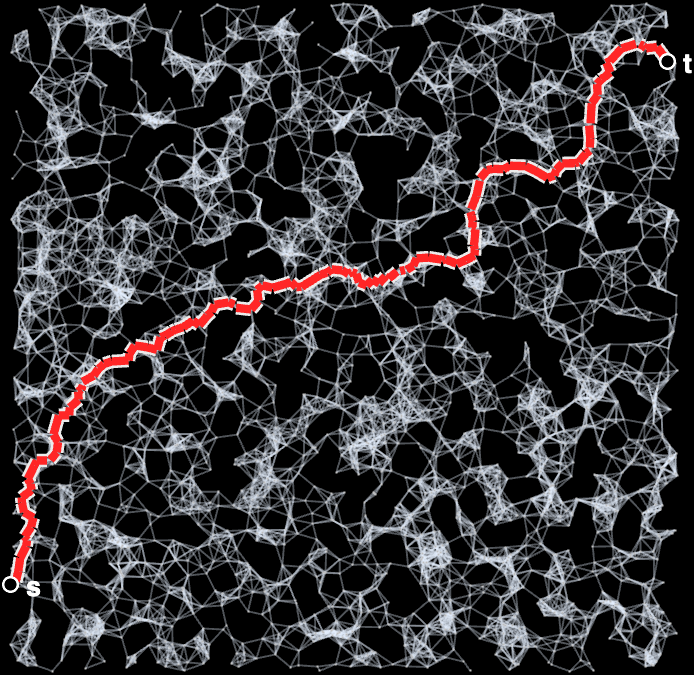}\vspace{-.7pt}
\includegraphics[width=\linewidth]{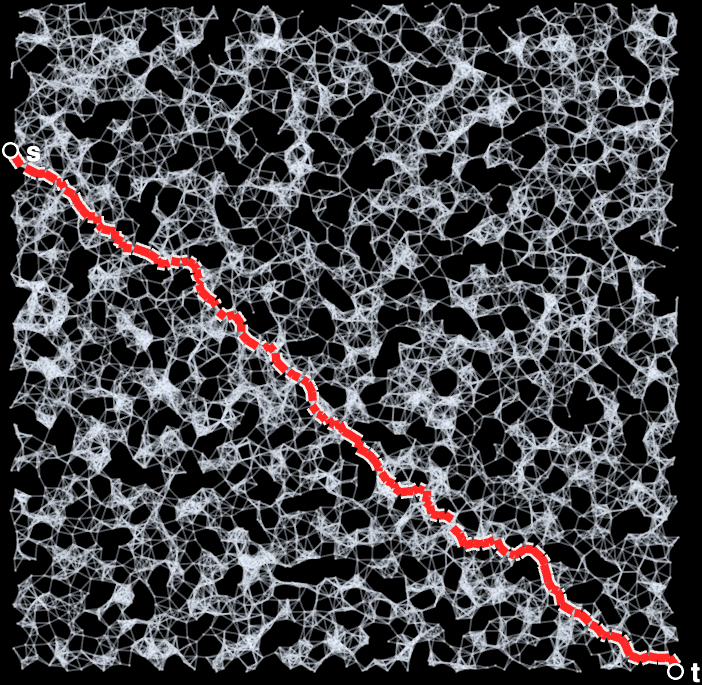}\vspace{-.7pt}
\includegraphics[width=\linewidth]{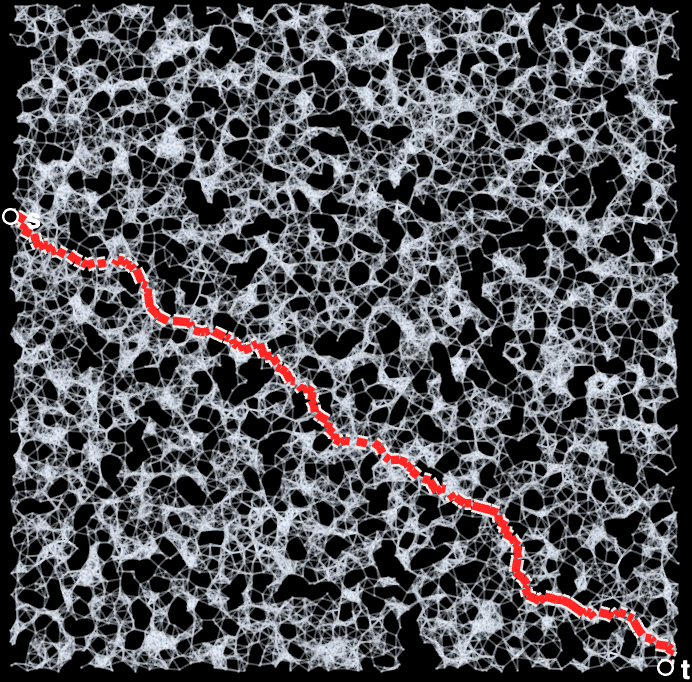}
\caption{ADMM}
\end{subfigure}\hspace{-2.3pt}
\begin{subfigure}{0.3\columnwidth}
\centering
\includegraphics[width=\linewidth]{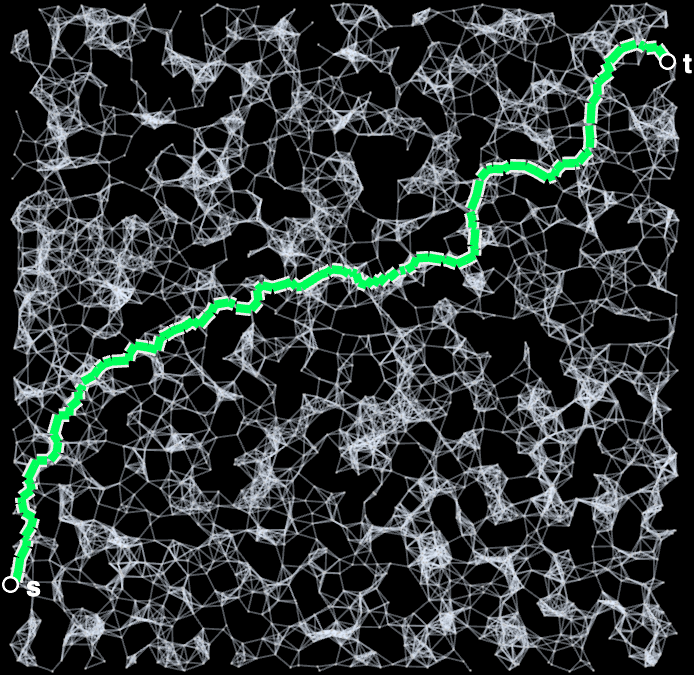}\vspace{-.7pt}
\includegraphics[width=\linewidth]{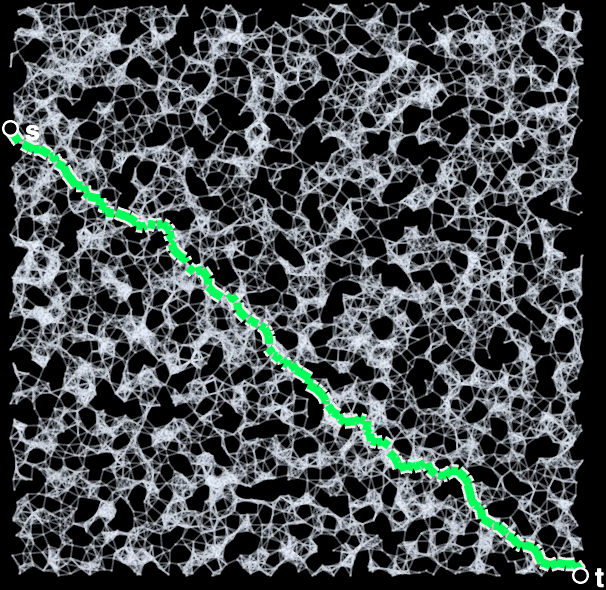}\vspace{-.7pt}
\includegraphics[width=\linewidth]{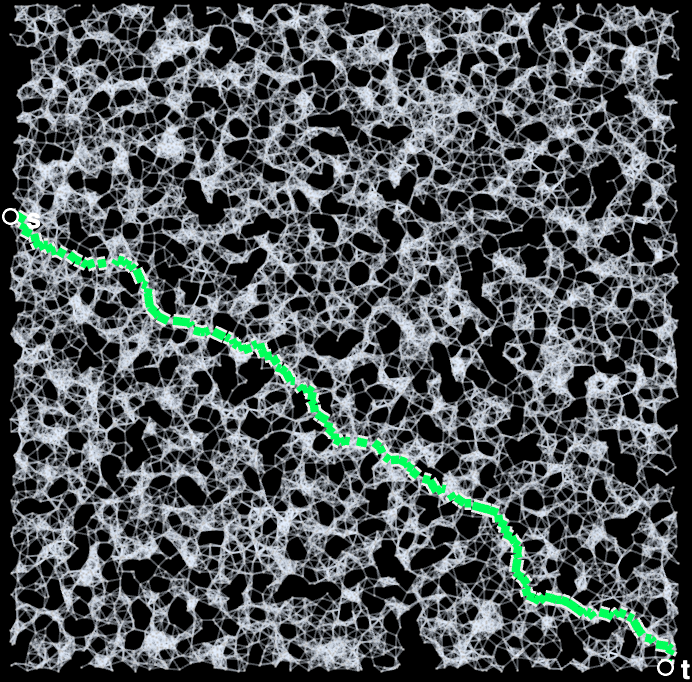}
\caption{InADMM}
\end{subfigure}
\caption{\blue Multi-hop routing on random geometric graphs. Shortest path from Dijkstra (yellow) and supports of the lasso solution $\beta$ recovered by ADMM (red) and InADMM (green) are overlaid. From top to bottom, graph size increases: top row $(n,m)=(2997,14149)$, middle row $(5997,30918)$, bottom row $(8989,48819)$.}
\label{fig:rgg}
\end{figure}
\black

\section{Concluding remarks}
The main contribution in this work is to point out that short paths in graphs, sought via formulation as a lasso problem, has computational and implementation advantages. Specifically, the use of ADMM algorithms for large graphs reduces computational complexity and allows for distributed implementation. Extension of the framework to one that can cope with negative weights is desirable, but at present, not available.

\section*{Acknowledgment}
This research has been supported in part by the NSF under ECCS-2347357, AFOSR under FA9550-24-1-0278, and ARO under W911NF-22-1-0292.

\appendix
\section*{Appendix}
\section{Proof of Proposition~\ref{prop:crossing}} \label{apdx:t-cross-derivation}
For simplicity, we drop the iteration subscript $k$ in our derivations. 
$D_\calA$ is the incidence matrix formed by the edges in the active set. The graph formed by $\calA$ consists of two disjoint trees $\Ts$, $\Tt$, and the set of isolated vertices $\Omega$. We decompose the rows of matrix $D_\calA$ into rows corresponding to these three subsets and express $D_\calA$ and $D_\calA^{+}$ according to
\begin{align*}
D_{\calA} = 
\begin{bmatrix}
D_{\calA^{(s)}} & \mathbf{0} \\
\mathbf{0} &  \mathbf{0} \\
\mathbf{0} &  D_{\calA^{(t)}}
\end{bmatrix},
\end{align*}
and
\begin{align*}
D_\calA^+ 
= 
\begin{bmatrix}
D_{\calA^{(s)}}^+ & \mathbf{0} & \mathbf{0} \\
\mathbf{0} & \mathbf{0} & D_{\calA^{(t)}}^+
\end{bmatrix},
\end{align*}
with $D_{\calA^{(s)}}, D_{\At}$ are the incidence matrix for the tree $\Ts$ and $\Tt$ respectively, and $\mathbf{0}$ represents all-zeros matrices of appropriate dimensions. We use this expression and Lemma~\ref{lem:incidence} to compute $a$ and $b$ by their definition~\eqref{eq:a-b}
\begin{align*}
a 
&=(Q_{\calA}^{T}Q_{\calA})^{+}Q_{\calA}^{T}y 
= Q_A^+y=W_\calA D_\calA^+y \\
&=
W_{\calA}
\begin{bmatrix}
D_{\calA^{(s)}}^+ &\mathbf{0} &\mathbf{0} \\
\mathbf{0}  &\mathbf{0}  & D_{\calA^{(t)}}^+  
\end{bmatrix}
\begin{bmatrix}
1 \\
\mathbf{0} \\
-1
\end{bmatrix}
\\
&= \displaystyle
\begin{bmatrix}
\displaystyle -\frac{1}{|\Ts|}W_{\As} P^{(s)} \mathbbm{1}_{|\Ts|} \\
\displaystyle +\frac{1}{|\Tt|}W_{\At} P^{(t)} \mathbbm{1}_{|\Tt|}
\end{bmatrix},
\end{align*}
where $P^{(s)}$ and $P^{(t)}$ are the path matrix for tree $\Ts$ and $\Tt$ respectively. Then,
\begin{align*}
b
&= (Q_\calA^TQ_\calA)^+ s \\
&= W_{\calA}D_{\calA}^{+} (D_{\calA}^{T})^{+}W_{\calA}s\\
& = \begin{bmatrix}
\displaystyle -W_{\As}\Big(P^{(s)} \calL^{(s)} -\frac{1}{|\Ts|}P^{(s)} \mathbbm{1}_s \mathbbm{1}_s^T\calL^{(s)}\Big)\\
\displaystyle +W_{\At}\Big(P^{(t)} \calL^{(t)} -\frac{1}{|\Tt|}P^{(t)} \mathbbm{1}_t  \mathbbm{1}_t^T\calL^{(t)}\Big)
\end{bmatrix},
\end{align*}
where $\calL^{(s)} \triangleq -(P^{(s)})^T W_{\calA^{(s)}} s$ is a vector of size $|\Ts|$ corresponding to vertices in the tree $\Ts$. The component of $\calL^{(s)}$,  corresponding to vertex $v \in \Ts$, is equal to $l_v^{(s)}$, i.e. the length of the path from  $v$ to the root $s$. The vector $\calL^{(t)} \triangleq (P^{(t)})^T W_{\calA^{(t)}} s$ has a similar interpretation, but for vertices of tree $\Tt$ {  and the difference of the definitions of $\calL^{(s)}$ and $\calL^{(t)}$ is resulted in the different signs of the two elements in $b$}. 

Putting the results for vectors $a$ and $b$ together, the ratio $a_j/b_j$ for $e_j \in \calA^{(s)}$ is
\begin{align*} 
\displaystyle \frac{a_j}{b_j} = \frac{\frac{1}{|\Ts|}w_j|R_j|}{w_j(\sum_{v\in R_j}l_v^{(s)} - \frac{|R_j|}{|\Ts|} \sum_{v\in \Ts} l_v^{(s)})},
\end{align*}
where $R_j$ is the set of non-zero components of the $j$th row of $P^{(s)}$ since the same edge $e_{j}$ shares the same direction along all the paths in a tree. This concludes our proof for $e_j \in \calA^{(s)}$. The derivation for $e_j \in \calA^{(t)}$ is similar.

\section{Proof of Proposition~\ref{prop:joining}}\label{apdx:t-join-derivation}
By definition of joining time~\eqref{eq:join}, for all $e_j \in \calE$, reads

\begin{align}\label{eq:t-join-apdx}
t^{\jn}_j = \frac{\frac{1}{w_j}D_j^T(Q_\calA a - y)}{\frac{1}{w_j}D_j^T(Q_\calA b) \pm 1} = {  \frac{D_j^T(Q_\calA a - y)}{D_j^T(Q_\calA b) \pm w_{j}}},
\end{align}
with the choice $\pm$ dictated by $t^{\jn}_j$ taking positive value.
We now obtain expressions for the terms in parentheses. First, the term $(Q_\calA a - y)$ in the numerator equals to
\begin{align}
D_\calA D_\calA^+ y - y 
&=  
\begin{bmatrix}
D_{\calA^{(s)}} & \mathbf{0} \\
\mathbf{0} &  \mathbf{0} \\
\mathbf{0} &  D_{\calA^{(t)}}
\end{bmatrix}
\begin{bmatrix}
D_{\calA^{(s)}}^+ & \mathbf{0} & \mathbf{0} \\
\mathbf{0} & \mathbf{0} &  D_{\calA^{(t)}}^+
\end{bmatrix}y - y \nonumber\\
&= 
\begin{bmatrix}
-\frac{1}{|\Ts|}\mathbbm{1}_s\\
\mathbb{0}_{|\Omega|}\\
+\frac{1}{|\Tt|}\mathbbm{1}_t
\end{bmatrix}, \label{eq:numerator}
\end{align}
where we used $$D D^+=I - \frac{1}{\mathbbm{1}^T\mathbbm{1}}\mathbbm{1}\mathbbm{1}^T$$ for incidence matrices $D_{\As}$ and $D_{\At}$ of the two disjoint trees{ ~\cite[Lemma 2.15]{bapat2010graphs}}, {  and $\mathbb{0}_{|\Omega|}$ denotes all-zero (column) vectors of size $|\Omega|$}. Second, the term $(Q_\calA b)$ in the denominator equals to $(D_\calA D_\calA^+ (D_\calA^T)^+ W_\calA s)$ and moreover, it equals to 
\begin{align}
(D_\calA^+)^T W_\calA s 
=
\begin{bmatrix}
\displaystyle -\calL^{(s)} + \frac{1}{|\Ts|} \mathbbm{1}_s \mathbbm{1}_s^T\calL^{(s)}\\
\mathbb{0}_{|\Omega|}\\
\displaystyle \calL^{(t)} -\frac{1}{|\Tt|} \mathbbm{1}_t  \mathbbm{1}_t^T\calL^{(t)}
\end{bmatrix}.
\end{align}
\blue where $\calL^{(s)}$ (resp. $\calL^{(t)}$) collects the root-to-vertex distances $l_v^{(s)}$ on $T^{(s)}$ (resp. $l_v^{(t)}$ on $T^{(t)}$).\black

Evaluating the expression in~\eqref{eq:t-join-apdx} for edge $e_j = (v_1,v_2)\in \calA^{c}$ and without loss of generality, assume the assigned direction is from $v_1$ to $v_2$, the joining time is categorized into three separate cases.

{\noindent \bf {\em i})} If $(v_1,v_2) \in \Omega^2  \cup  {T^{(s)}}^2 \cup {T^{(t)}}^2$, i.e., the two ends of $e_{j}$ belongs to or isolated from the trees, we have $t_{j}^{\jn} = 0$ by the numerator~\eqref{eq:numerator} and the definition of $D_{j}^{T}$.

{\noindent \bf {\em ii})} If $(v_1,v_2) \in (T^{(s)}\times\Omega)\cup (T^{(t)}\times\Omega)$, Without loss of generality, assume $v_1 \in T^{(s)}$ \blue and $v_2\in\Omega$, using the blocks above, we have
\[
D_j^\top(Q_\calA a-y)=\tfrac{1}{|\Ts|},\qquad
D_j^\top(Q_\calA b)=\,l^{(s)}_{v_2}-\frac{1}{|\Ts|}\sum_{v\in T^{(s)}} l_v^{(s)}.
\]\black
The joining now reads
\begin{align*}
t_j^\jn 
&=\big(|\Ts|l^{(s)}_{v_1} - \mathbbm{1}^{T}\calL^{(s)}_{v} \pm |\Ts|w_{j}\big)^{-1}\\
&=\Big(|T^{(s)}|(l^{(s)}_{v_1}+w_{j}) - \mathbbm{1}^{T}\calL^{(s)}_{v}\Big)^{-1}\\
&{\blue =\Big(|\Ts|\,l^{(s)}_{v_2}-\sum_{v\in T^{(s)}} l_v^{(s)}\Big)^{-1}}.
\end{align*}
(The expression is equivalent to using $l^{(s)}_{v_2}=l^{(s)}_{v_1}+w_j$ upon joining.)

\blue
{\noindent \bf {\em iii})} If $(v_1,v_2)\in T^{(s)}\times T^{(t)}$, then
\[
D_j^\top(Q_\calA a-y)=\frac{1}{|\Ts|}+\frac{1}{|\Tt|},
\quad
D_j^\top(Q_\calA b)=\Big(l^{(s)}_{v_1}-\overline{l^{(s)}}_{T^{(s)}}\Big)+\Big(l^{(t)}_{v_2}-\overline{l^{(t)}}_{T^{(t)}}\Big),
\]
where $\overline{l^{(s)}}_{T^{(s)}}=\displaystyle \frac{1}{|\Ts|}\sum_{v\in T^{(s)}} l_v^{(s)}$ and
$\overline{l^{(t)}}_{T^{(t)}}=\displaystyle \frac{1}{|\Tt|}\sum_{v\in T^{(t)}} l_v^{(t)}$.
At the connection step, for the edge that realizes the maximum joining time we have
$l^{(s)}_{v_1}+w_j+l^{(t)}_{v_2}=l_t^{(s)}$ (i.e., the $s$–$t$ distance),
hence the denominator in~\eqref{eq:t-join-apdx} equals
\[
l_t^{(s)}-\overline{l^{(s)}}_{T^{(s)}}-\overline{l^{(t)}}_{T^{(t)}}
=\frac{\Delta}{|\Ts|\,|\Tt|},
\]
with
\[
\Delta
=|\Ts|\,|\Tt|\,l_t^{(s)}-|\Tt|\sum_{v\in T^{(s)}} l_v^{(s)}-|\Ts|\sum_{v\in T^{(t)}} l_v^{(t)}.
\]
Therefore
\[
t_j^\jn=\frac{\frac{1}{|\Ts|}+\frac{1}{|\Tt|}}{\frac{\Delta}{|\Ts|\,|\Tt|}}
=\frac{|\Ts|+|\Tt|}{\Delta}.
\]
\black

\section{Proof of Lemma~\ref{lemma:uni}} \label{apdx:proof-lemma-3.2}
The proof is based on the sufficient condition for the uniqueness of the lasso solution\cite[Lemma 2]{tibshirani2013lasso}, which states that ``{\em For any $\beta$,~$Q$, and $\lambda>0$, if} $\mbox{null}(Q_\calA)=\{0\}$ ({\em or equivalently if }$\mbox{rank}(Q_{\calA})$ = $|\calA|$), {\em the lasso solution and the active set $\calA$ are always unique.''} 
Thus, we prove the lasso solution  $\beta(\lambda)$ is unique for every $\lambda > 0$ by showing that rank$(Q_{\calA})$ = $|\calA|$ is true for every $Q_{\calA}$.

In Section~\ref{sec:relation}, we showed that the active set forms two disjoint trees (Lemma~\ref{lemma:addactive} and \ref{lemma:cross}). Hence, 
\begin{align*}
Q_{\calA} = 
\begin{bmatrix}
Q_{\calA^{(s)}}, &Q_{\calA^{(t)}}    
\end{bmatrix} = 
\begin{bmatrix}
D_{\calA^{(s)}} W_{\calA^{(s)}}^{-1}, &D_{\calA_{}^{(t)}} W_{\calA_{}^{(t)}}^{-1}  
\end{bmatrix},
\end{align*}
with $D_{\calA_{k}^{(t)}}$ and $D_{\calA_{k}^{(t)}}$ are incidence matrices of two trees positively weighted by the weighted matrices $W_{\calA_{k}^{(s),(t)}}$. The {  null space} {(\em kernel)} of the incidence matrix of a tree is empty because there is no cycle in the tree by its definition, so the rank of its incidence matrix is equal to the number of columns, i.e., the rank of $Q_\calA$ is equal to $|\calA|$. Therefore, the sufficient condition for the unique lasso solution is satisfied.

\bibliographystyle{plain}
\bibliography{ref}

\end{document}